\documentclass[aap,press]{apt} \volno{58} \partno{}
\usepackage{apt-arxiv}
\usepackage[colorlinks=true,linkcolor=blue,urlcolor=blue,citecolor=blue,anchorcolor=blue]{hyperref}

\paperno{}
\authornames{T.~SHI AND M.~XIA}
\shorttitle{Scale--shape factorization of Poisson--Voronoi cell volume}

\usepackage{amsmath,amssymb,mathtools,bm,graphicx,placeins}

\newcommand{\R}{\mathbb{R}}
\newcommand{\Sph}{\mathbb{S}}
\newcommand{\PP}{\Phi_{\lambda}}
\newcommand{\vol}{\operatorname{vol}}
\newcommand{\Haus}{\mathcal{H}}
\newcommand{\ind}{\mathbf{1}}
\newcommand{\dd}{\mathrm{d}}
\newcommand{\norm}[1]{\left\lVert #1\right\rVert}
\newcommand{\inner}[2]{\left\langle #1,#2\right\rangle}
\newcommand{\bfx}{\bm{x}}
\newcommand{\bfy}{\bm{y}}
\newcommand{\bfu}{\bm{u}}
\newcommand{\bfp}{\bm{p}}
\newcommand{\bfd}{\bm{\delta}}
\newcommand{\E}{\mathbb{E}}
\newcommand{\Var}{\operatorname{Var}}

\usepackage{siunitx}

\begin{document}

\title{Exact Scale--Shape Factorization of the Typical Poisson--Voronoi Cell Volume in Arbitrary Dimension}
\hypersetup{
pdftitle={Exact Scale--Shape Factorization of the Typical Poisson--Voronoi Cell Volume in Arbitrary Dimension}, pdfauthor={Tian Shi and Minghua Xia}, pdfsubject={Palm-typical Poisson--Voronoi cell volume}, pdfkeywords={Poisson point process, Poisson--Voronoi tessellation, Palm distribution, Voronoi flower, scale--shape factorization} }

\authorone{Tian Shi}
\authoronetwo[Sun Yat-sen University]{Minghua Xia}
\addressone{School of Electronics and Information Technology, Sun Yat-sen University, Guangzhou 510006, China}
\emailtwo{xiamingh@mail.sysu.edu.cn}

\begin{abstract}
Despite more than six decades of research, a tractable closed-form distribution for the typical Poisson--Voronoi cell volume remains unknown beyond one dimension. Building on the classical complementary-theorem structure for the Poisson--Voronoi fundamental region, we develop an explicit configuration-space factorization of the Palm-typical cell volume for a tessellation generated by a stationary Poisson point process of intensity \(\lambda>0\) in \(\mathbb R^d\), \(d\geq1\). Conditional on the number \(k\) of effective facets, the Voronoi flower volume is a \(\operatorname{Gamma}(k,\operatorname{rate}=\lambda)\) scale variable independent of the effective-neighbour configuration normalized to have unit flower volume. Mapping this normalized configuration to its cell-to-flower volume ratio \(A_k\) gives the conditional cell-volume representation as the product of the classical Gamma scale and a bounded geometric shape factor. From this representation, we derive exact mixture formulae, transform and moment identities, and criteria characterizing when the conditional cell-volume laws are Gamma. We also distinguish shape-factor variability within facet-number strata from mixing across strata as two sources of departure from a single Gamma law and recover the universal bound \(0<A_k\leq2^{-d}\). For the lower tail, we reduce critical negative moments of the shape factor to inverse-volume integrals over normalized configuration spaces. Under explicit critical-integrability and higher-facet summability assumptions, the density of the intensity-normalized cell volume has leading order \(y^d\) as \(y\downarrow0\). The framework recovers the one-dimensional distribution, admits explicit planar coordinates, and supports numerical evaluation of the mixture representation in dimensions two through four.
\end{abstract}



\section{Introduction}
\label{sec:introduction}

The geometric idea of partitioning space by nearest nuclei has a history that extends well beyond the modern theory of random tessellations. It dates back to Dirichlet's study of quadratic forms in 1850~\cite{Dirichlet1850}, and Voronoi developed it systematically in arbitrary dimensions in 1908~\cite{Voronoi1908}. When the nuclei are generated by a stationary Poisson point process, this deterministic geometric construction becomes the random Poisson--Voronoi tessellation. Gilbert's classical 1962 study~\cite{Gilbert1962} initiated its modern stochastic theory, and Miles~\cite{Miles1970} later advanced it.

Despite more than six decades of research on the stochastic model, no tractable closed-form characterization of the typical Poisson--Voronoi cell-volume distribution is known beyond one dimension. The theory is treated systematically in standard monographs on random tessellations and stochastic geometry~\cite{Moller1994,Okabe2000,SchneiderWeil2008}, and exact results are available for selected geometric functionals, facet counts, moments, and asymptotic regimes~\cite{AlishahiSharifitabar2008,Calka2002,Hilhorst2005,Muche2005, MucheBallani2011}. The dual Poisson--Delaunay tessellation provides a useful contrast: its typical-simplex volumes admit dimension-explicit formulae and have applications to wireless-network modelling~\cite{8358978}. For Poisson--Voronoi cell volumes, however, practical studies continue to rely heavily on Monte Carlo simulation, moment information, and Gamma-type empirical approximations~\cite{HindeMiles1980,Kumar1992,Tanemura2003, FerencNeda2007,Marthinsen1996}.

A recent survey of typical-cell volume distributions~\cite{11655630} emphasizes this analytical asymmetry. Beyond one dimension, available descriptions of the Poisson--Voronoi volume consist primarily of planar integral representations, moment identities, simulation benchmarks, and parametric approximations. In the planar case, Calka derived exact formulae for the principal characteristics of the typical cell by conditioning on its number of sides and integrating over neighbouring-point distances and angular spacings~\cite{Calka2003}. The unconditional area law is then obtained through an infinite mixture over the side count. These formulae are exact, but their planar, configuration-dependent integral form does not yield a tractable unconditional distribution or a framework that extends directly across spatial dimensions.

An important partial structure is nevertheless known. Zuyev established that, conditional on the number of hyperfaces, the volume of the Poisson--Voronoi fundamental region is Gamma distributed and used its containment relation with the cell to obtain a cell-volume tail bound~\cite{Zuyev1992}. M{\o}ller and Zuyev subsequently placed this result in a general complementary-theorem framework and showed that the fundamental-region volume is conditionally independent of the normalized determining configuration \cite{MollerZuyev1996}. What remains unresolved is how the cell volume is obtained from this Gamma scale through the configuration-dependent cell-to-flower ratio and how that ratio produces departures from a single Gamma law.

Let \(d\geq1\) denote the ambient spatial dimension, and let \(\lambda>0\) denote the intensity of the underlying stationary Poisson point process in \(\mathbb R^d\). A finite collection of effective neighbouring nuclei determines a candidate cell \(P\). Its Voronoi flower is
\[
  \Delta(P)=\bigcup_{z\in P}\overline B(z,\norm z),
\]
which coincides with the classical fundamental region and is precisely the exclusion region in which an additional nucleus would alter the candidate cell. Let
\[
  Q=\vol_d(\Delta(P))
  \qquad\text{and}\qquad
  G=\vol_d(P)
\]
denote the flower and candidate-cell volumes, respectively. Under a common positive dilation of the neighbouring nuclei, both \(Q\) and \(G\) are homogeneous of degree \(d\). Consequently, \(Q\) can be used as a radial scale, whereas the cell-to-flower ratio \(G/Q\) depends only on the normalized effective-neighbour configuration.

We derive a finite-dimensional configuration formula for the joint law of the Palm-cell volume and its number of effective facets. Conditional on \(k\) effective facets, where \(k\geq d+1\), the multivariate Slivnyak--Mecke formula contributes the Poisson factor \(\lambda^k e^{-\lambda Q}\)~\cite{LastPenrose2017}. A polar-coordinate decomposition adapted to the homogeneous functional \(Q\) contributes the radial factor \(t^{k-1}\), where \(t=Q\). Their product is the density kernel of a \(\operatorname{Gamma}(k,\operatorname{rate}=\lambda)\) distribution, while the remaining angular measure gives an explicit probability law on normalized effective-neighbour configurations. This construction provides a direct configuration-space realization of the scale--configuration separation.

Mapping the normalized configuration to the cell-to-flower ratio expresses the Palm-cell volume, conditional on the facet number, as the product of the Gamma flower-volume scale and an independent bounded shape factor. This formulation identifies how the exclusion-volume scale transforms into the cell-volume law. Unlike Calka's planar coordinate formulae~\cite{Calka2003}, the resulting framework applies in arbitrary dimension and separates variability of the cell-to-flower ratio within facet-number strata from mixing across those strata. It does not, however, provide a closed-form evaluation of the cell-volume density, because the induced shape-factor laws remain generally implicit.

The factorization yields exact mixture formulae for the distribution and density, transform and moment identities, and a characterization of when the conditional cell-volume law is Gamma with shape equal to the facet number. The normalized configurations are represented by probability measures on finite-dimensional configuration spaces, and their images under the cell-to-flower ratio give the corresponding shape-factor laws. The classical geometric inclusion between the cell and its flower gives the universal bound
\[
  0<A_k\leq 2^{-d}.
\]
Together, these results provide an analytical representation of the cell-volume law and a basis for numerical evaluation and approximation.

The same representation provides a geometric route to the lower tail. For each facet number \(k\), the flower-volume factors cancel at negative moment order \(-k\), reducing the relevant negative moments of the shape factor to inverse-volume integrals over normalized configuration spaces. We derive small-volume and tubular-neighbourhood criteria for the finiteness of these integrals and express the higher-facet summability requirement in geometric terms. The minimal-facet stratum has positive probability and consists of simplices. Under two logically distinct assumptions---critical integrability on the minimal-facet stratum and summability over the higher-facet strata---we establish
\[
  f_{Y_d}(y)
  \sim
  \frac{p_{d,d+1}}{\Gamma(d+1)}
  \E\!\left[A_{d+1}^{-(d+1)}\right]y^d,
  \qquad y\downarrow0,
\]
where \(Y_d=\lambda V_{d,\lambda}\). The leading coefficient therefore separates into the minimal-facet probability and a negative moment of the corresponding shape factor.

The one-dimensional cell-volume distribution is recovered exactly, while the planar case admits explicit coordinates connecting the general construction to classical integral representations. Numerical experiments in dimensions two through four provide diagnostics for the conditional Gamma scale law, scale--configuration independence, the support bound, and Monte Carlo evaluation of the exact mixture representation.

Section~\ref{sec:geometric-measure-setup} develops the Palm geometry, the Voronoi flower, the effective-neighbour configuration formula, and the homogeneous radial normalization. Section~\ref{sec:main-results-consequences} establishes the configuration-space cell-volume factorization and derives its distributional, moment, Gamma-characterization, and conditional lower-tail consequences. Section~\ref{sec:low-dimensional} presents the low-dimensional specializations and numerical evaluation, and Section~\ref{sec:conclusion} concludes.

\section{Geometric and measure-theoretic setup}
\label{sec:geometric-measure-setup}

This section develops the geometric and measure-theoretic framework underlying the exact scale--shape factorization. We first introduce the Palm-typical Voronoi cell, its effective-neighbour configurations, and the associated Voronoi flower. We then establish the exclusion identity and the finite-dimensional configuration representation of the Palm cell. Finally, using the homogeneity of the flower volume, we construct the normalized configuration space representing shape and derive a radial decomposition that separates each admissible configuration into its flower-volume scale and normalized configuration.

\subsection{Palm cell, candidate configurations, and the flower}
\label{Section:II-A}

The scale--shape factorization developed in the next section requires a finite-dimensional description of the Palm-typical Voronoi cell. We begin by introducing the Palm cell, its finite effective-neighbour configurations, and the flower functional that encodes the corresponding Poisson exclusion volume.

Let $\PP$ be a stationary Poisson point process in space $\R^d$ with intensity $\lambda>0$. Slivnyak's theorem identifies the Palm-typical cell with the cell of an added nucleus at the origin \cite{LastPenrose2017}:
\begin{equation*}
  C_0
  =
  \left\{
  z\in\R^d:\norm z\leq\norm{z-x}
  \ \text{for every }x\in\PP
  \right\}.
\end{equation*}
We denote the volume of \(C_0\) by
\begin{equation*}
  V_{d,\lambda}=\vol_d(C_0).
\end{equation*}
For $x\neq0$, put
\begin{equation*}
  H(x)
  =
  \left\{
  z\in\R^d:
  \inner{z}{x}\leq\frac{\norm{x}^2}{2}
  \right\}.
\end{equation*}
Since
\begin{equation*}
  \norm z\leq\norm{z-x}
  \quad\Longleftrightarrow\quad
  \inner{z}{x}\leq\frac{\norm{x}^2}{2},
\end{equation*}
the Palm cell has the half-space representation
\begin{equation*}
  C_0=\bigcap_{x\in\PP}H(x).
\end{equation*}

Although the Palm cell is generated by the entire Poisson process, only finitely many nuclei determine its facets. To isolate this finite-dimensional geometry, consider an ordered $k$-tuple $\bfx=(x_1,\ldots,x_k)$ of distinct nonzero points and define the associated candidate cell by
\begin{equation*}
  P(\bfx)=\bigcap_{i=1}^kH(x_i).
\end{equation*}
Let $\mathcal A_{d,k}$ be the set of tuples for which $P(\bfx)$ is bounded, full-dimensional, and has exactly $k$ facets, with the $i$th constraint generating one facet. Lower-dimensional degeneracies, such as zero-volume facets, coincident supporting hyperplanes, nontransversal vertices, or non-simple changes of face lattice, are isolated below and do not affect the configuration integrals. Write
\begin{equation*}
  G(\bfx)=\vol_d(P(\bfx)).
\end{equation*}
Every bounded full-dimensional polytope in $\R^d$ has at least $d+1$ facets, so $\mathcal A_{d,k}$ is empty for $k<d+1$.

The cell volume \(G(\bfx)\) describes the size of the candidate cell, but it does not by itself determine whether the candidate survives the remaining Poisson points. The relevant object is the associated exclusion region, commonly called the Voronoi flower. For a compact convex set \(L\subset\R^d\) containing the origin, define
\begin{equation*}
  \Delta(L)=\bigcup_{z\in L}\overline B(z,\norm z).
\end{equation*}
When \(L\) is a polytope, this union is already generated by its vertices:
\begin{equation*}
  \Delta(L)
  =
  \bigcup_{v\in\operatorname{vert}(L)}
  \overline B(v,\norm v).
\end{equation*}
Thus, for a Poisson--Voronoi candidate cell, \(\Delta(L)\) coincides with the \emph{fundamental region} used by Zuyev~\cite{Zuyev1992} and by M{\o}ller and Zuyev~\cite{MollerZuyev1996}. For a candidate configuration \(\bfx\), set
\begin{equation*}
  Q(\bfx)=\vol_d(\Delta(P(\bfx))).
\end{equation*}
As shown below, this volume is the natural scale variable because it determines the Poisson exclusion probability for the candidate configuration.

The flower volume admits a support-function representation that we will use repeatedly. Let
\begin{equation*}
  h_L(u)=\sup_{z\in L}\inner{u}{z},
  \qquad u\in\Sph^{d-1}.
\end{equation*}
Since $0\in L$, we have $h_L(u)\geq0$. For $r>0$,
\begin{align*}
  ru\in\Delta(L)
  &\Longleftrightarrow
  \text{there exists }z\in L\text{ such that }\norm{ru-z}^2\leq\norm z^2\\
  &\Longleftrightarrow
  \text{there exists }z\in L\text{ such that }r^2-2r\inner{u}{z}\leq0\\
  &\Longleftrightarrow
  r\leq2h_L(u).
\end{align*}
Thus the flower is a star body whose radial function is determined by the support function of $L$. In the notation of support and radial functions of convex bodies \cite{Schneider2014},
\begin{equation*}
  \rho_{\Delta(L)}(u)=2h_L(u).
\end{equation*}
Polar integration therefore gives
\begin{equation}
  \label{eq:flower-support}
  \vol_d(\Delta(L))
  =
  \frac{2^d}{d}
  \int_{\Sph^{d-1}}h_L(u)^d\,\dd S(u).
\end{equation}
For comparison, the radial function of the candidate cell is
\begin{equation*}
  \rho_{P(\bfx)}(u)
  =
  \min_{i:\inner{u}{x_i}>0}
  \frac{\norm{x_i}^2}{2\inner{u}{x_i}},
\end{equation*}
and hence
\begin{equation*}
  G(\bfx)
  =
  \frac1d
  \int_{\Sph^{d-1}}\rho_{P(\bfx)}(u)^d\,\dd S(u).
\end{equation*}

For illustration, Fig.~\ref{fig:VoronoiFlower} shows the above objects in the planar case. The blue polygon represents the candidate cell \(P\), while the red region represents its flower \(\Delta(P)\). In view of \(\rho_{\Delta(P)}(u)=2h_P(u)\), the flower boundary is the radial expansion determined by the support function of \(P\). In two dimensions, this boundary is composed of circular arcs: the arc between two adjacent effective neighbours \(x_i\) and \(x_{i+1}\) lies on the circumcircle of the triangle with vertices \(0\), \(x_i\), and \(x_{i+1}\). Thus, the figure also illustrates the two geometric quantities entering the cell-to-flower volume ratio, namely \(G=\operatorname{area}(P)\) and \(Q=\operatorname{area}(\Delta(P))\).

\begin{figure}[!t]
	\centering
	\includegraphics[width=0.7\linewidth]{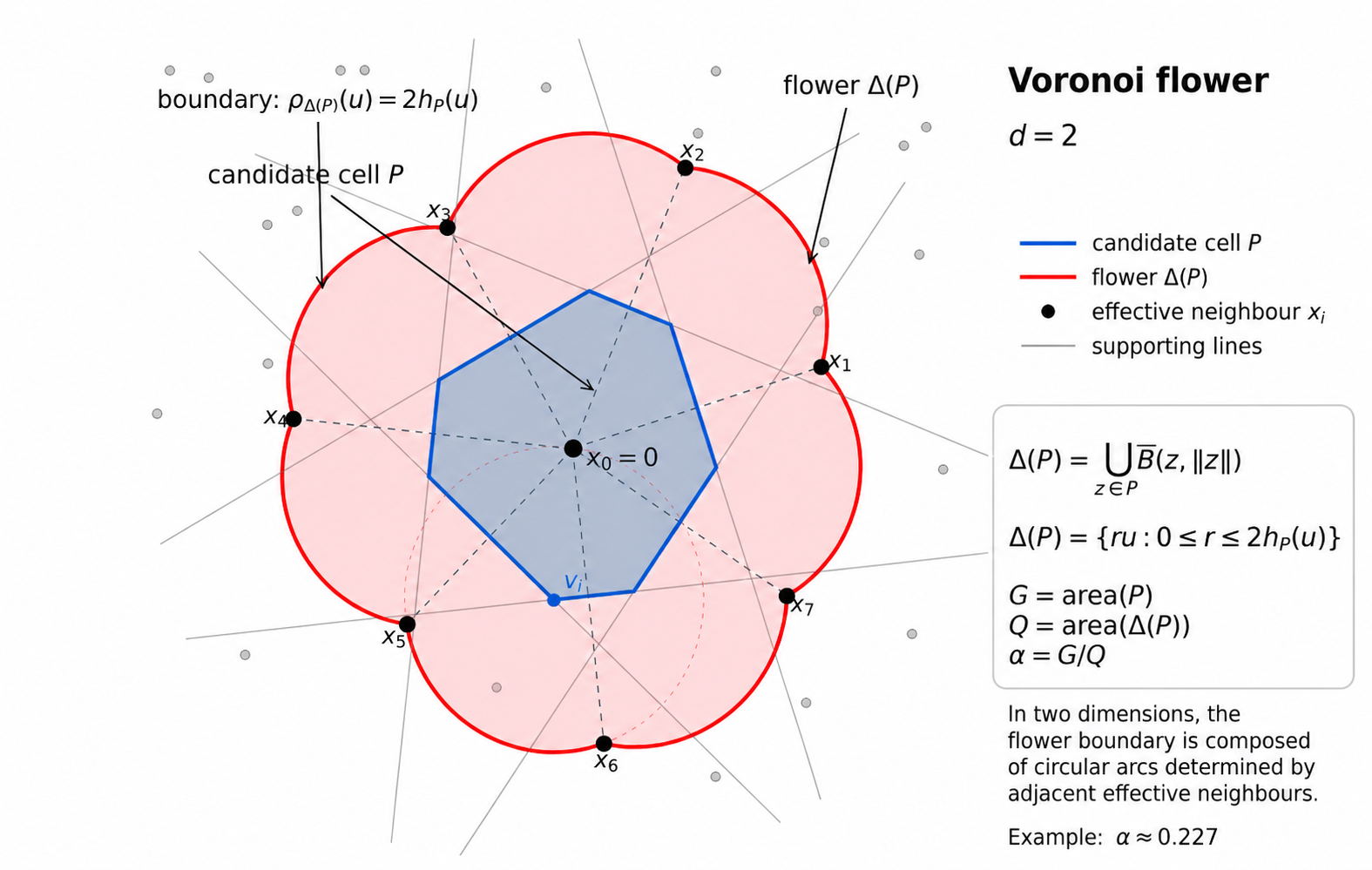}
	\caption{Planar illustration of a Palm candidate cell and its Voronoi flower. The blue polygon is the candidate cell \(P\), the red region is the flower \(\Delta(P)\), and the black points are the effective neighbouring nuclei.}
	\label{fig:VoronoiFlower}
\end{figure}

Before applying measure-theoretic tools, we isolate the regular configuration space on which the combinatorial type of the candidate cell and the associated geometric functionals vary smoothly. The exceptional configurations are lower-dimensional and hence irrelevant for the Lebesgue configuration integrals used below.

\begin{lemma}[Admissible configurations and regular strata]
  \label{lem:admissible-measurable}
Fix \(d\geq1\) and \(k\geq1\), and identify \((\R^d)^k\) with \(\R^{dk}\). The admissible configuration set \(\mathcal A_{d,k}\) is semialgebraic, and hence Borel. Moreover, there exists a Lebesgue-null semialgebraic set \(\mathcal N_{d,k}\subset\mathcal A_{d,k}\) such that
  \[
    \mathcal A_{d,k}^{\rm reg}
    :=
    \mathcal A_{d,k}\setminus\mathcal N_{d,k}
    =
    \bigsqcup_{\ell=1}^{L_{d,k}}\mathcal S_\ell
  \]
is a finite disjoint union of connected smooth semialgebraic strata. The sets \(\mathcal N_{d,k}\), \(\mathcal A_{d,k}^{\rm reg}\), and \(\mathcal S_\ell\) may be chosen invariant under positive scaling. On every stratum \(\mathcal S_\ell\), the face lattice of \(P(\bfx)\) is fixed, every vertex of \(P(\bfx)\) depends smoothly on \(\bfx\), and \(\bfx\mapsto G(\bfx)\) is smooth.
\end{lemma}

\begin{proof}
Put \(m=dk\). If \(\mathcal A_{d,k}=\varnothing\), the result is immediate; hence assume below that \(\mathcal A_{d,k}\neq\varnothing\).

For every \(d\)-element subset \(I=\{i_1,\ldots,i_d\}\subset\{1,\ldots,k\}\), define
  \[
    M_I(\bfx)
    =
    (x_{i_1},\ldots,x_{i_d})^T,
    \qquad
    b_I(\bfx)
    =
    \frac12
    \bigl(
      \norm{x_{i_1}}^2,\ldots,\norm{x_{i_d}}^2
    \bigr)^T.
  \]
On the determinant chart
  \[
    \mathcal U_I
    =
    \{\bfx\in\R^m:\det M_I(\bfx)\neq0\},
  \]
the intersection of the \(d\) supporting hyperplanes indexed by \(I\) is
  \[
    v_I(\bfx)=M_I(\bfx)^{-1}b_I(\bfx).
  \]
Thus \(v_I\) is a rational, and therefore smooth, function on \(\mathcal U_I\).

The condition that \(v_I(\bfx)\) is feasible for \(P(\bfx)\) is
  \[
    \inner{v_I(\bfx)}{x_j}
    \leq
    \frac{\norm{x_j}^2}{2},
    \qquad j=1,\ldots,k.
  \]

On each of the two semialgebraic charts \(\{\det M_I>0\}\) and \(\{\det M_I<0\}\), clearing denominators turns these relations into polynomial inequalities. Hence the sets on which \(v_I\) is a feasible vertex, as well as the sets on which it lies on any additional supporting hyperplane, are semialgebraic.

The remaining conditions defining \(\mathcal A_{d,k}\) are semialgebraic as well. Boundedness of \(P(\bfx)\) is equivalent to
\[
  \{z\in\R^d:
    \inner{z}{x_i}\leq0,\ i=1,\ldots,k\}
  =
  \{0\},
\]
or, equivalently, to \(x_1,\ldots,x_k\) positively spanning \(\R^d\). Full dimensionality is equivalent to the existence of a point satisfying all defining inequalities strictly. Finally, the assertion that the \(i\)th constraint defines a facet can be expressed by the existence of \(d\) affinely independent points satisfying the \(i\)th inequality with equality and all remaining inequalities weakly. These are first-order conditions involving finitely many polynomial equalities and inequalities. The Tarski--Seidenberg theorem \cite[Chapter~2]{BochnakCosteRoy1998} therefore shows that their projections onto the configuration variables are semialgebraic. Consequently, \(\mathcal A_{d,k}\) is semialgebraic and hence Borel.

We next construct the regular strata. Since admissibility is invariant under positive scaling, it suffices to work on the Euclidean configuration sphere. Set
\[
  \Theta_{d,k}
  =
  \mathcal A_{d,k}\cap\Sph^{m-1}.
\]
Its closure in \(\Sph^{m-1}\) is a compact semialgebraic set. Consider the finite family of semialgebraic subsets encoding membership in \(\Theta_{d,k}\), the determinant signs of all \(M_I\), the feasibility of all candidate vertices \(v_I\), and all vertex--hyperplane incidences. By the compatible semialgebraic Whitney-stratification theorem \cite{Shiota2005}, \(\overline{\Theta_{d,k}}\) admits a finite smooth semialgebraic stratification compatible with this family. Refining the strata into connected components, if necessary, preserves finiteness because semialgebraic sets have finitely many connected components \cite[Chapter~2]{BochnakCosteRoy1998}.

Compatibility implies that, on each stratum contained in \(\Theta_{d,k}\), the determinant signs, feasible-vertex relations, and vertex--hyperplane incidences are fixed. Consequently, the complete vertex--facet incidence structure, and hence the face lattice of \(P(\theta)\), is fixed. Every vertex is represented there by one of the smooth rational maps \(v_I\).

Let \(\Theta_{d,k}^{\rm reg}\) be the union of the strata of \(\Theta_{d,k}\) having dimension \(m-1\), and let
  \[
    \mathcal N_{d,k}
    =
    \{r\theta:r>0,\ 
      \theta\in\Theta_{d,k}\setminus\Theta_{d,k}^{\rm reg}\}.
  \]
The omitted spherical strata have dimension at most \(m-2\); their conical extensions therefore have dimension at most \(m-1\) and hence zero \(m\)-dimensional Lebesgue measure. Thus \(\mathcal N_{d,k}\) is Lebesgue null. By construction, both \(\mathcal N_{d,k}\) and its complement in \(\mathcal A_{d,k}\) are invariant under positive scaling.

Finally, take the positive conical extension of each connected component of \(\Theta_{d,k}^{\rm reg}\). These extensions are the strata \(\mathcal S_1,\ldots,\mathcal S_{L_{d,k}}\). On each such stratum, the face lattice is fixed, and all vertices depend smoothly on \(\bfx\). A polytope with fixed face lattice may be triangulated by a fixed combinatorial rule. Its volume is then a finite sum of simplex volumes, each given locally by the absolute value of a determinant of smooth vertex coordinates. After a further finite semialgebraic refinement fixing the determinant signs, these simplex volumes, and therefore
  \[
    G(\bfx)=\vol_d(P(\bfx)),
  \]
are smooth on every stratum.
\end{proof}

Lemma~\ref{lem:admissible-measurable} provides the finite-dimensional regularity required below: after removal of a conic null set, the admissible configuration space is partitioned into finitely many smooth semialgebraic strata on which the face lattice is fixed, and both the vertices and cell volume vary smoothly. The radial normalization, however, is determined by the flower volume \(Q\), not by \(G\). We therefore use the support-function representation \eqref{eq:flower-support} to transfer the stratum-wise regularity of the vertices to \(Q\). For the subsequent polar decomposition, measurability and positivity of \(Q\) suffice; no smoothness of the level set \(\{Q=1\}\) is required.

\begin{lemma}[Measurability of the flower-volume map]
  \label{lem:Q-measurable}
The map $Q(\bfx)=\vol_d(\Delta(P(\bfx)))$ is locally Lipschitz on every regular stratum of $\mathcal A_{d,k}^{\rm reg}$. Consequently, $Q$ is Borel measurable on $\mathcal A_{d,k}^{\rm reg}$.
\end{lemma}

\begin{proof}
Fix a regular stratum. By Lemma~\ref{lem:admissible-measurable}, the vertices of $P(\bfx)$ are given there by a fixed finite collection of smooth maps $v_1(\bfx),\ldots,v_N(\bfx)$. Hence
  \[
    h_{P(\bfx)}(u)=\max_{1\leq j\leq N}\inner{u}{v_j(\bfx)},
    \qquad u\in\Sph^{d-1}.
  \]
On every compact subset $K_0$ of the stratum, the vertex maps are uniformly bounded and Lipschitz. Thus, for some $L_{K_0}<\infty$,
  \[
    \bigl|h_{P(\bfx)}(u)-h_{P(\bfy)}(u)\bigr|
    \leq L_{K_0}\norm{\bfx-\bfy},
    \qquad \bfx,\bfy\in K_0,\quad u\in\Sph^{d-1}.
  \]
The support functions are also uniformly bounded on $K_0\times\Sph^{d-1}$. Using \eqref{eq:flower-support} and the local Lipschitz bound for $a\mapsto a^d$ on bounded intervals gives
  \[
    |Q(\bfx)-Q(\bfy)|\leq C_{K_0}\norm{\bfx-\bfy},
    \qquad \bfx,\bfy\in K_0,
  \]
for some $C_{K_0}<\infty$. Therefore $Q$ is locally Lipschitz on the stratum. Since the regular configuration space is a finite union of Borel strata, $Q$ is Borel measurable.
\end{proof}

\subsection{Exclusion identity and configuration formula}
\label{Section:II-B}

The preceding subsection introduced the flower as a geometric object associated with a candidate cell. We now show that it also has a direct probabilistic interpretation. Specifically, the flower is exactly the exclusion region whose emptiness determines whether a candidate configuration survives as the Palm cell. This observation converts the geometry of candidate configurations into explicit Poisson probabilities and forms the bridge between Palm geometry and finite-dimensional configuration integrals.

\begin{lemma}[Complete exclusion region]
  \label{lem:exclusion}
Fix $\bfx\in\mathcal A_{d,k}$ and put
  \[
    P=P(\bfx).
  \]
A nonzero additional nucleus $y$ changes $P$ on a set with nonempty interior if and only if
  \begin{equation*}
    y\in\operatorname{int}\Delta(P).
  \end{equation*}
Consequently,
  \begin{equation}
    \label{eq:void-probability}
    \mathbb P
    \{
      \PP\cap\operatorname{int}\Delta(P)=\varnothing
    \}
    =
    \exp[-\lambda Q(\bfx)] .
  \end{equation}
\end{lemma}

\begin{proof}
Adding the nucleus $y$ replaces $P$ by $P\cap H(y)$.  A point $z\in P$ is removed precisely when
  \begin{equation*}
    \norm{z-y}<\norm z .
  \end{equation*}
For fixed $z$, this strict inequality is equivalent to $y\in B(z,\norm z)$.  Therefore, the set of nuclei that remove at least one point of $P$ is
  \begin{equation}
    \label{eq:open-flower}
    \bigcup_{z\in P}B(z,\norm z).
  \end{equation}
Since $0$ is an interior point of $P$, the support function $h_P(u)$ is positive for every $u\in\Sph^{d-1}$.  For $y=ru$ with $r>0$ and $u\in\Sph^{d-1}$,
  \begin{align*}
    y\in\bigcup_{z\in P}B(z,\norm z)
    &\Longleftrightarrow
      \text{there exists }z\in P
      \text{ such that } \norm{ru-z}^2<\norm z^2  \\
    &\Longleftrightarrow
      \text{there exists }z\in P
      \text{ such that } r<2\inner{u}{z}  \\
    &\Longleftrightarrow
      r<2h_P(u).
  \end{align*}
Thus the set in \eqref{eq:open-flower} is exactly $\operatorname{int}\Delta(P)\setminus\{0\}$.

If $y\in\operatorname{int}\Delta(P)$, then the strict inequality $\norm{z-y}<\norm z$ holds for some $z\in P$.  By continuity, it holds in a neighbourhood of $z$, and the intersection of that neighbourhood with the full-dimensional polytope $P$ has nonempty interior.  Hence $P\cap H(y)$ differs from $P$ on a set with nonempty interior.  If $y\notin\operatorname{int}\Delta(P)$, then no point of $P$ satisfies the strict inequality.  The additional half-space can only create boundary contact and therefore cannot remove a subset of $P$ with nonempty interior.

Finally, $\Delta(P)$ has the radial representation
  \begin{equation*}
    \Delta(P)=\{ru:u\in\Sph^{d-1},\ 0\leq r\leq2h_P(u)\}.
  \end{equation*}
Since $P$ is compact, $h_P$ is Lipschitz on $\Sph^{d-1}$; hence $\partial\Delta(P)$ is contained in finitely many Lipschitz graphs in polar charts and has zero $d$-dimensional Lebesgue measure.  Therefore
  \begin{equation*}
    \vol_d(\operatorname{int}\Delta(P))
    =\vol_d(\Delta(P))
    =Q(\bfx).
  \end{equation*}
The Poisson void probability formula gives \eqref{eq:void-probability}.
\end{proof}

\begin{remark}[Flower volume as an exclusion measure]
Lemma~\ref{lem:exclusion} identifies the Voronoi flower, equivalently the fundamental region, as the exclusion region associated with the candidate cell. Specifically,
\[
  Q(\bfx)=\vol_d(\Delta(P(\bfx)))
\]
is the Lebesgue measure of the set in which the presence of an additional Poisson nucleus would alter \(P(\bfx)\). Hence, the absence of all remaining nuclei from its interior is precisely the condition ensuring that the candidate cell remains unchanged, up to null boundary modifications. The corresponding Poisson void probability is therefore
\[
  \exp\{-\lambda Q(\bfx)\}.
\]
Thus, \(Q(\bfx)\) is the natural exclusion-volume scale associated with the candidate configuration.
\end{remark}

Let $\mathcal F_{d-1}(P)$ denote the set of facets of a full-dimensional polytope $P$, and define
\[
N(P):=\#\mathcal F_{d-1}(P).
\]
Write
\[
K=N(C_0)
\]
for the number of effective facets of the Palm cell. The next proposition transforms Palm probabilities into finite-dimensional configuration integrals over admissible neighbour configurations.

The effective neighbours of the Palm cell form a finite subprocess characterized by the absence of all remaining Poisson nuclei from its associated fundamental, or exclusion region. This finite-subprocess viewpoint underlies the general Poisson-process construction of M{\o}ller and Zuyev~\cite[Section~3]{MollerZuyev1996}. The following proposition gives a direct ordered-neighbour version tailored to the joint distribution of the Palm-cell volume and its number of effective facets.

\begin{proposition}[Configuration representation of the Palm cell]
\label{prop:configuration}
For every Borel set \(B\subset(0,\infty)\) and every \(k\ge d+1\),
\begin{equation}
\label{eq:configuration}
\mathbb P
\{
V_{d,\lambda}\in B,\,
K=k
\}
=
\frac{\lambda^k}{k!}
\int_{\mathcal A_{d,k}}
\ind\{G(\bfx)\in B\}
e^{-\lambda Q(\bfx)}
\,\dd\bfx.
\end{equation}
\end{proposition}

\begin{proof}
Let \(\Phi_{\lambda,\neq}^k\) denote the ordered \(k\)-tuples of distinct points of \(\PP\). For a locally finite set \(\varphi\) and an ordered tuple \(\bfx=(x_1,\ldots,x_k)\), define
\begin{equation*}
  I_B(\bfx,\varphi)
  =
  \ind\{\bfx\in\mathcal A_{d,k}\}
  \ind\{G(\bfx)\in B\}
  \ind\{\varphi\cap\operatorname{int}\Delta(P(\bfx))=\varnothing\}.
\end{equation*}

Outside the null set of degenerate Palm cells, the event \(\{K=k\}\) has a unique unordered set of \(k\) effective neighbouring nuclei. Each of these nuclei generates exactly one facet of \(C_0\), and every facet is generated by one of them. Therefore the ordered sum
\begin{equation*}
  \frac1{k!}\sum_{\bfx\in\Phi_{\lambda,\neq}^k}
  I_B(\bfx,\PP\setminus\{x_1,\ldots,x_k\})
\end{equation*}
is exactly the indicator of the event \(\{V_{d,\lambda}\in B,K=k\}\).

Indeed, if \(\{K=k\}\) holds and \(\{x_1,\ldots,x_k\}\) is the effective-neighbour set, then \(P(\bfx)=C_0\) for each ordering of this set, \(G(\bfx)=V_{d,\lambda}\), and Lemma~\ref{lem:exclusion} implies that no remaining point of \(\PP\) lies in \(\operatorname{int}\Delta(P(\bfx))\). The \(k!\) orderings are counted, and the factor \(1/k!\) yields a single contribution. Conversely, if an ordered tuple contributes to the sum, then no point of the remaining process can cut \(P(\bfx)\) on a set with nonempty interior. Hence \(P(\bfx)\) is the Palm cell up to a null boundary modification. Since \(\bfx\in\mathcal A_{d,k}\), this cell has exactly \(k\) effective facets and volume \(G(\bfx)\in B\).

Taking expectations and applying the multivariate Slivnyak--Mecke formula~\cite{LastPenrose2017} gives
\begin{align*}
  \mathbb P\{V_{d,\lambda}\in B,K=k\}
  &=
  \frac1{k!}\E\sum_{\bfx\in\Phi_{\lambda,\neq}^k}
  I_B(\bfx,\PP\setminus\{x_1,\ldots,x_k\})\\
  &=
  \frac{\lambda^k}{k!}
  \int_{(\R^d)^k}
  \E\bigl[I_B(\bfx,\PP)\bigr]\,\dd\bfx.
\end{align*}
For fixed \(\bfx\in\mathcal A_{d,k}\), the remaining expectation is the void probability of \(\operatorname{int}\Delta(P(\bfx))\), namely \(e^{-\lambda Q(\bfx)}\) by Lemma~\ref{lem:exclusion}. Outside \(\mathcal A_{d,k}\), the integrand is zero. This proves \eqref{eq:configuration}.
\end{proof}

The representation \eqref{eq:configuration} is an explicit configuration-space specialization of the classical finite-subprocess/exclusion-region framework. Its additional feature is that it retains the candidate-cell functional \(G(\bfx)\), thereby giving the joint law of cell volume and facet number rather than only that of the determining subprocess or its exclusion volume. The Poisson void probability produces the factor \(e^{-\lambda Q(\bfx)}\), while the remaining integral records the geometry of admissible effective-neighbour configurations. The subsequent radial normalization yields a product measure for the flower-volume scale and normalized configuration. Mapping the latter through \(G/Q\) produces the cell-volume factorization and the associated mixture representation developed below.

\subsection{Scale--shape normalization and radial decomposition}
\label{Section:II-C}

The flower volume is a positive homogeneous functional on the admissible configuration space. This homogeneity separates every admissible configuration into its flower-volume scale and a normalized configuration. We now construct the resulting measure on the unit-flower-volume configuration space and derive the radial decomposition used below.

For $r>0$, the constraint generated by $rx_i$ may be written as
\begin{equation*}
  \inner{z}{x_i}\leq r\frac{\norm{x_i}^2}{2}.
\end{equation*}
It follows that
\begin{equation}
  \label{eq:homogeneity}
  P(r\bfx)=rP(\bfx),\qquad
  G(r\bfx)=r^dG(\bfx),\qquad
  Q(r\bfx)=r^dQ(\bfx).
\end{equation}
Thus, positive scaling preserves the normalized geometry of a configuration while changing only its scale.

The homogeneity of \(Q\) allows the scale of an admissible configuration to be removed by normalizing its flower volume to one. Accordingly, define the normalized configuration space
\begin{equation*}
  \Sigma_{d,k}
  =
  \{\bfu\in\mathcal A_{d,k}^{\rm reg}:Q(\bfu)=1\}.
\end{equation*}
Every positive dilation orbit in \(\mathcal A_{d,k}^{\rm reg}\) intersects \(\Sigma_{d,k}\) in exactly one point; thus, the normalization removes dilation while retaining the configuration's orientation, ordering, and relative geometry. Put $n=dk$, let
\[
  \Theta_{d,k}=\mathcal A_{d,k}^{\rm reg}\cap\Sph^{n-1},
\]
and define the normalization map
\[
  \psi(\theta)=Q(\theta)^{-1/d}\theta,
  \qquad \theta\in\Theta_{d,k}.
\]
The positivity of $Q$ and Lemma~\ref{lem:Q-measurable} make $\psi$ measurable, and $\psi(\theta)\in\Sigma_{d,k}$.  We define the normalized-configuration measure as the push-forward
\begin{equation}
  \label{eq:mu-def}
  \mu_{d,k}
  =
  \psi_\#\left(
    \frac1d Q(\theta)^{-k}
    \Haus^{n-1}\big|_{\Theta_{d,k}}
  \right).
\end{equation}
This definition uses only ordinary Euclidean polar coordinates; in particular, it does not require $Q$ to be differentiable on the fixed level set $Q=1$.

Every \(\bfx\in\mathcal A_{d,k}^{\rm reg}\) lies on exactly one positive dilation orbit and admits the unique representation
\begin{equation*}
  t=Q(\bfx),\qquad
  \bfu=t^{-1/d}\bfx\in\Sigma_{d,k},\qquad
  \bfx=t^{1/d}\bfu.
\end{equation*}
The cell-to-flower volume ratio is invariant under positive dilation and therefore depends only on the normalized configuration. Define
\begin{equation}
  \label{eq:alpha-def}
  \alpha(\bfu)=G(\bfu)
  =
  \frac{\vol_d(P(\bfu))}
       {\vol_d(\Delta(P(\bfu)))}.
\end{equation}
Thus, \(\bfu\) represents the normalized effective-neighbour configuration, whereas \(\alpha(\bfu)\) is the scalar cell-to-flower volume ratio determined by that configuration. We refer to \(\alpha(\bfu)\), and later to its random counterpart \(A_k\), as the \emph{shape factor}. The equality on the right of \eqref{eq:alpha-def} follows because \(Q(\bfu)=1\) on \(\Sigma_{d,k}\). Homogeneity gives
\begin{equation}
  \label{eq:G-scale-shape}
  G(t^{1/d}\bfu)=t\alpha(\bfu).
\end{equation}

The following identity separates integration over the admissible configuration space into integration over radial scale and normalized configuration.
\begin{lemma}[Scale--shape radial decomposition]
  \label{lem:coarea}
For every nonnegative measurable function $h$ on $\mathcal A_{d,k}$,
  \begin{equation}
    \label{eq:coarea}
    \int_{\mathcal A_{d,k}}h(\bfx)\,\dd\bfx
    =
    \int_{\Sigma_{d,k}}\int_0^\infty
    h(t^{1/d}\bfu)t^{k-1}\,\dd t\,\mu_{d,k}(\dd\bfu).
  \end{equation}
\end{lemma}

\begin{proof}
The admissible set is a cone: if $\bfx\in\mathcal A_{d,k}^{\rm reg}$ and $r>0$, then $r\bfx\in\mathcal A_{d,k}^{\rm reg}$.  The omitted exceptional set has Lebesgue measure zero.  Euclidean polar coordinates in $\R^n$ therefore give
  \begin{equation*}
    \int_{\mathcal A_{d,k}}h(\bfx)\,\dd\bfx
    =\int_{\Theta_{d,k}}\int_0^\infty
      h(r\theta)r^{n-1}\,\dd r\,\dd\Haus^{n-1}(\theta).
  \end{equation*}
For fixed $\theta$, set $t=r^dQ(\theta)$.  Then
  \[
    r\theta=t^{1/d}\psi(\theta),\qquad
    r^{n-1}\,\dd r
    =\frac1d t^{n/d-1}Q(\theta)^{-n/d}\,\dd t
    =\frac1d t^{k-1}Q(\theta)^{-k}\,\dd t.
  \]
Substitution followed by the push-forward definition \eqref{eq:mu-def} proves \eqref{eq:coarea}.  Tonelli's theorem \cite[Chapter~2]{Folland1999} applies because $h$ is nonnegative.
\end{proof}

\begin{remark}[Scale--configuration interpretation of the radial decomposition]
The decomposition \eqref{eq:coarea} is a polar-coordinate formula adapted to the homogeneous functional \(Q\). The flower volume \(t=Q(\bfx)\) is the radial scale, while \(\mu_{d,k}\) describes configurations normalized to have unit flower volume. The factor \(t^{k-1}\) is the radial Jacobian associated with this normalization. At this stage, the decomposition is purely geometric and measure-theoretic. The Gamma law and the probabilistic independence of scale and normalized configuration arise after it is combined with the Poisson factor \(\lambda^k e^{-\lambda t}\).
\end{remark}

The preceding construction equips the normalized configuration space with a finite measure whenever the following quantity is finite. Set
\begin{equation}
  \label{eq:p-def}
  p_{d,k}=\frac1k\mu_{d,k}(\Sigma_{d,k}).
\end{equation}
The factorization theorem below identifies this quantity with the facet-number probability, thereby justifying the notation. When $p_{d,k}>0$, define the probability measure
\begin{equation*}
  \nu_{d,k}(\dd\bfu)
  =
  \frac{\mu_{d,k}(\dd\bfu)}{kp_{d,k}}
\end{equation*}
and its push-forward under $\alpha$ by
\begin{equation*}
  \eta_{d,k}=\alpha_\#\nu_{d,k}.
\end{equation*}

If \(p_{d,k}=0\), neither a normalized-configuration law nor its induced shape-factor law is needed; later formulae either condition on a positive-probability stratum or multiply the corresponding term by $p_{d,k}$.

The constructions above complete the geometric preparation for the configuration-space factorization. Section~\ref{sec:main-results-consequences} combines the flower-volume scale with the normalized effective-neighbour configuration and maps the latter through the cell-to-flower ratio, yielding a Gamma-scale product representation of the Palm-cell volume.

\section{Main results and consequences}
\label{sec:main-results-consequences}

The geometric preparation in Section~\ref{sec:geometric-measure-setup} now yields the main probabilistic factorization. Conditional on the number of effective facets, the flower volume is a Gamma scale variable independent of the normalized effective-neighbour configuration.

For $k$ with $p_{d,k}>0$, define a probability measure on ordered admissible configurations by
\begin{equation}
  \label{eq:conditional-configuration-law}
  \mathsf P_{d,k,\lambda}(\dd\bfx)
  =
  \frac{\lambda^k}{k!p_{d,k}}
  e^{-\lambda Q(\bfx)}
  \ind\{\bfx\in\mathcal A_{d,k}\}\,\dd\bfx.
\end{equation}
Its normalization follows from Proposition~\ref{prop:configuration} with $B=(0,\infty)$. Equivalently, it is the law obtained by giving the effective neighbours of the Palm cell an independent uniform ordering, conditional on $K=k$. All geometric maps used below are invariant under permutations of this ordering.

The Gamma law and scale--configuration independence in the following theorem recover the Poisson--Voronoi specialization of the complementary theorem of M{\o}ller and Zuyev~\cite[Theorem~3 and Example~5(ii)]{MollerZuyev1996}. Our formulation additionally identifies the normalized-configuration measure and maps it through the cell-to-flower ratio to obtain the Palm-cell volume factorization.

\begin{theorem}[Configuration-space factorization of the Palm-cell volume]
\label{thm:factorization}
Let \(K=N(C_0)\). Then the quantities defined in \eqref{eq:p-def} are the facet-number probabilities:
\begin{equation*}
  \mathbb P(K=k)=p_{d,k},
  \qquad k\geq d+1.
\end{equation*}

For every \(k\) with \(p_{d,k}>0\), let \(\bfx\) have distribution \(\mathsf P_{d,k,\lambda}\), and define
\begin{equation*}
  T_k=Q(\bfx),
  \qquad
  U_k=Q(\bfx)^{-1/d}\bfx\in\Sigma_{d,k}.
\end{equation*}
Then their joint distribution is the product measure
\begin{equation}
  \label{eq:theorem-product-measure}
  \mathsf P_{d,k,\lambda}
  \{T_k\in\dd t,U_k\in\dd\bfu\}
  =
  \frac{\lambda^k}{\Gamma(k)}
  t^{k-1}e^{-\lambda t}\,\dd t\,
  \nu_{d,k}(\dd\bfu).
\end{equation}
In particular,
\begin{equation*}
  T_k\sim
  \operatorname{Gamma}\bigl(k,\operatorname{rate}=\lambda\bigr),
  \qquad
  U_k\sim\nu_{d,k},
  \qquad
  T_k\perp U_k.
\end{equation*}

Writing
\[
A_k=\alpha(U_k)=\frac{G(U_k)}{Q(U_k)}=G(U_k),
\]
the conditional Palm-cell volume satisfies
\begin{equation}
  \label{eq:V-factorization}
  V_{d,\lambda}\mid\{K=k\}
  \overset{\mathrm d}=T_kA_k,
\end{equation}
where
\[
A_k\sim\eta_{d,k},
\qquad
T_k\perp A_k.
\]

Equivalently, for the intensity-normalized cell volume
\[
Y_d=\lambda V_{d,\lambda},
\]
one may construct \(K\) with \(\mathbb P(K=k)=p_{d,k}\) and, conditionally on \(K=k\), independent random variables
\[
A_k\sim\eta_{d,k},
\qquad
Z_k\sim\operatorname{Gamma}(k,\operatorname{rate}=1),
\]
such that
\begin{equation}
  \label{eq:Y-factorization}
  Y_d\overset{\mathrm d}=A_KZ_K.
\end{equation}
Values assigned on strata with \(p_{d,k}=0\) are immaterial.
\end{theorem}

\begin{proof}
Substitute Lemma~\ref{lem:coarea} into Proposition~\ref{prop:configuration}. For every Borel set \(B\subset(0,\infty)\),
\begin{align}
  \mathbb P\{V_{d,\lambda}\in B,K=k\}
  &=
  \frac{\lambda^k}{k!}
  \int_{\Sigma_{d,k}}\int_0^\infty
  \ind\{t\alpha(\bfu)\in B\}
  e^{-\lambda t}t^{k-1}\,\dd t\,
  \mu_{d,k}(\dd\bfu).
  \label{eq:joint-scale-shape}
\end{align}

Taking \(B=(0,\infty)\) gives
\begin{align*}
  \mathbb P(K=k)
  &=
  \frac{\lambda^k}{k!}\mu_{d,k}(\Sigma_{d,k})
  \int_0^\infty e^{-\lambda t}t^{k-1}\,\dd t\\
  &=
  \frac{\lambda^k}{k!}\mu_{d,k}(\Sigma_{d,k})
  \frac{\Gamma(k)}{\lambda^k}\\
  &=
  \frac1k\mu_{d,k}(\Sigma_{d,k})
  =
  p_{d,k}.
\end{align*}
In particular, \(\mu_{d,k}(\Sigma_{d,k})<\infty\), and the normalization defining \(\nu_{d,k}\) is legitimate whenever \(p_{d,k}>0\).

Now assume \(p_{d,k}>0\). For Borel sets \(D\subset(0,\infty)\) and \(E\subset\Sigma_{d,k}\), the definition of \(\mathsf P_{d,k,\lambda}\), together with Lemma~\ref{lem:coarea}, gives
\begin{align*}
  \mathsf P_{d,k,\lambda}\{T_k\in D,U_k\in E\}
  &=
  \frac{\lambda^k}{k!p_{d,k}}
  \int_E\int_D
  e^{-\lambda t}t^{k-1}\,\dd t\,
  \mu_{d,k}(\dd\bfu).
\end{align*}
Using \(\mu_{d,k}=kp_{d,k}\nu_{d,k}\) and \(k!=k\Gamma(k)\), we obtain
\begin{equation}
  \label{eq:conditional-product-measure}
  \mathsf P_{d,k,\lambda}
  \{T_k\in\dd t,U_k\in\dd\bfu\}
  =
  \frac{\lambda^k}{\Gamma(k)}
  t^{k-1}e^{-\lambda t}\,\dd t\,
  \nu_{d,k}(\dd\bfu).
\end{equation}
Thus \(T_k\) has the \(\operatorname{Gamma}(k,\operatorname{rate}=\lambda)\) distribution and is independent of \(U_k\), whose distribution is \(\nu_{d,k}\).

Since \(A_k=\alpha(U_k)\), it follows that \(A_k\sim\eta_{d,k}\) and \(A_k\) is independent of \(T_k\). Moreover,
\[
  G(\bfx)
  =
  G(t^{1/d}\bfu)
  =
  t\alpha(\bfu),
\]
which proves \eqref{eq:V-factorization}. Finally, multiplying \eqref{eq:V-factorization} by \(\lambda\) and setting \(Z_k=\lambda T_k\) gives
\[
  Z_k\sim\operatorname{Gamma}(k,\operatorname{rate}=1),
\]
and hence \eqref{eq:Y-factorization}.
\end{proof}

Theorem~\ref{thm:factorization} has two layers. Its scale--configuration component gives an explicit configuration-space realization of the classical complementary-theorem result: conditional on \(K=k\), the flower volume \(T_k=Q\) is Gamma distributed and is independent of the normalized effective-neighbour configuration \(U_k\). The additional cell-volume component applies the geometric map
\[
\alpha(U_k)=\frac{G(U_k)}{Q(U_k)}
\]
to that normalized configuration. Consequently, the conditional Palm-cell volume is the product of the classical Gamma exclusion-volume scale and the induced cell-to-flower shape factor. This product representation is the starting point for the cell-volume mixture, transform, moment, and lower-tail results developed below.

The only scalar geometric information entering this product is the shape factor
\[
  A_k=\alpha(U_k),
\]
namely the cell-to-flower volume ratio determined by the normalized configuration. The following bound shows that this random variable has a universal compact upper support, independent of both the intensity and the facet number.

The inclusion \(2P\subseteq\Delta(P)\), used by Zuyev to derive a cell-volume tail bound~\cite[Section~4]{Zuyev1992}, gives the following uniform bound. We provide a direct support-function proof in the present notation.

\begin{lemma}[Cell-to-flower volume bound]
  \label{lem:alpha-bound}
For every admissible candidate cell \(P\),
  \begin{equation}
    \label{eq:alpha-bound}
    0<
    \frac{\vol_d(P)}{\vol_d(\Delta(P))}
    \leq 2^{-d}.
  \end{equation}
Consequently, every nontrivial shape-factor law \(\eta_{d,k}\) is supported on \((0,2^{-d}]\). The constant \(2^{-d}\) is universal and is sharp in the larger class of convex bodies, as shown by centered Euclidean balls.
\end{lemma}

\begin{proof}
Let \(\rho_P\) and \(h_P\) denote the radial and support functions of \(P\). Since \(\rho_P(u)u\in P\),
\begin{equation*}
  \rho_P(u)
  =
  \inner{u}{\rho_P(u)u}
  \leq h_P(u),
  \qquad u\in\Sph^{d-1}.
\end{equation*}
Together with the identity \(\rho_{\Delta(P)}(u)=2h_P(u)\), this gives
\begin{equation*}
  \rho_{\Delta(P)}(u)
  \geq 2\rho_P(u),
  \qquad u\in\Sph^{d-1}.
\end{equation*}
Equivalently,
\[
  2P\subseteq\Delta(P).
\]
Polar integration therefore yields
\begin{align*}
  \vol_d(\Delta(P))
  &=
  \frac1d\int_{\Sph^{d-1}}
  \rho_{\Delta(P)}(u)^d\,\dd S(u)\\
  &\geq
  \frac{2^d}{d}\int_{\Sph^{d-1}}
  \rho_P(u)^d\,\dd S(u)\\
  &=
  2^d\vol_d(P).
\end{align*}
Therefore,
\[
  \frac{\vol_d(P)}{\vol_d(\Delta(P))}
  \leq 2^{-d}.
\]
The ratio is strictly positive because \(P\) is bounded and full-dimensional.
\end{proof}

Combining Lemma~\ref{lem:alpha-bound} with the cell-volume factorization in Theorem~\ref{thm:factorization} yields explicit mixture formulae for the distribution and density of the Palm-cell volume.

\begin{corollary}[CDF and PDF]
  \label{cor:cdf-pdf}
For \(v>0\),
  \begin{equation}
    \label{eq:main-cdf}
    F_{V_{d,\lambda}}(v)
    =
    \sum_{k=d+1}^\infty p_{d,k}
    \int_{(0,2^{-d}]}
    \frac{\gamma(k,\lambda v/a)}{\Gamma(k)}
    \,\eta_{d,k}(\dd a),
  \end{equation}
and
  \begin{equation}
    \label{eq:main-pdf}
    f_{V_{d,\lambda}}(v)
    =
    \sum_{k=d+1}^\infty p_{d,k}
    \int_{(0,2^{-d}]}
    \frac{\lambda^kv^{k-1}}{\Gamma(k)a^k}
    \exp\left(-\frac{\lambda v}{a}\right)
    \,\eta_{d,k}(\dd a),
  \end{equation}
where
  \[
    \gamma(k,z)=\int_0^z s^{k-1}e^{-s}\,\dd s
  \]
is the lower incomplete Gamma function. Terms with \(p_{d,k}=0\) are interpreted as zero.
\end{corollary}

\begin{remark}[Intensity scaling]
Writing $C_{0,\lambda}$ for the Palm cell at intensity $\lambda$, the mapping and scaling properties of Poisson processes \cite{LastPenrose2017,SchneiderWeil2008} give
  \begin{equation*}
    C_{0,\lambda}\overset{\mathrm d}=\lambda^{-1/d}C_{0,1},
    \qquad
    V_{d,\lambda}\overset{\mathrm d}=\lambda^{-1}V_{d,1}.
  \end{equation*}
Thus $Y_d=\lambda V_{d,\lambda}$, $p_{d,k}$, and $\eta_{d,k}$ do not depend on $\lambda$, as is also evident from the normalized factorization.
\end{remark}

\subsection{Transforms, moments, and Gamma diagnostics}
\label{Section:III-A}

The factorization reduces transforms and moments of the typical-cell volume to corresponding quantities of the shape factor \(A_k\). Consequently, all departures from a pure Gamma law are encoded by the distribution of $A_k$ and by the mixing over the facet number.

\begin{proposition}[Transforms and moments]
  \label{prop:transforms-moments}
For every $s\geq0$,
  \begin{equation}
    \label{eq:laplace-unconditional}
    \E[e^{-sY_d}]
    =
    \sum_{k=d+1}^\infty
    p_{d,k}\,
    \E\bigl[(1+sA_k)^{-k}\bigr],
  \end{equation}
with the convention that terms with $p_{d,k}=0$ contribute zero. Moreover, for every integer $r\geq1$,
  \begin{equation}
    \label{eq:moment-unconditional}
    \E[Y_d^r]
    =
    \sum_{k=d+1}^\infty
    p_{d,k}\,
    (k)_r\,\E[A_k^r],
  \end{equation}
where the identity is understood in \([0,\infty]\), and \((k)_r=\Gamma(k+r)/\Gamma(k)\) denotes the rising factorial, also called the Pochhammer symbol.
\end{proposition}

\begin{proof}
Fix $k$ with $p_{d,k}>0$.  By Theorem~\ref{thm:factorization}, conditionally on $K=k$,
  \begin{equation*}
    Y_d=A_kZ_k,
    \qquad
    Z_k\sim\operatorname{Gamma}(k,\operatorname{rate}=1),
  \end{equation*}
and $Z_k$ is independent of $A_k$.  Therefore, conditionally on $A_k=a$,
  \begin{equation*}
    Y_d\mid\{K=k,A_k=a\}
    \overset{\rm d}=aZ_k.
  \end{equation*}
The Laplace transform of $Z_k$ gives
  \begin{equation*}
    \E[e^{-sY_d}\mid K=k,A_k=a]
    =
    \E[e^{-saZ_k}]
    =
    (1+sa)^{-k}.
  \end{equation*}
Averaging over $A_k$ yields
  \begin{equation*}
    \E[e^{-sY_d}\mid K=k]
    =
    \E\bigl[(1+sA_k)^{-k}\bigr].
  \end{equation*}
A final average over $K$ gives \eqref{eq:laplace-unconditional}.

Similarly, for every integer $r\geq1$,
  \begin{equation*}
    \E[Y_d^r\mid K=k,A_k=a]
    =
    a^r\E[Z_k^r]
    =
    a^r\frac{\Gamma(k+r)}{\Gamma(k)}.
  \end{equation*}
Hence
  \begin{equation*}
    \E[Y_d^r\mid K=k]
    =
    (k)_r\E[A_k^r],
    \qquad
    (k)_r=\frac{\Gamma(k+r)}{\Gamma(k)}.
  \end{equation*}
Averaging over $K$ gives \eqref{eq:moment-unconditional}.  Since all terms are nonnegative, Tonelli's theorem justifies the identities in the extended sense.
\end{proof}

Theorem~\ref{thm:factorization} reduces the distributional analysis of the typical-cell volume to the shape factor \(A_k\). Proposition~\ref{prop:transforms-moments} makes this reduction explicit by expressing the Laplace transforms and moments as averages over the shape-factor laws \(\eta_{d,k}\). Consequently, once these laws are understood, the corresponding distributional properties of the typical-cell volume follow immediately.

For $r\geq1$ and $k$ with $p_{d,k}>0$, write
\begin{equation*}
  m_{r,k}:=\E[A_k^r],
\end{equation*}
whenever the expectation is finite; terms with $p_{d,k}=0$ are understood as zero in the following sums.

The standard mean-volume identity $\E[V_{d,\lambda}]=1/\lambda$ for the typical cell of a stationary Poisson--Voronoi tessellation \cite{Moller1994,SchneiderWeil2008} gives the normalization
\begin{equation*}
  \sum_{k=d+1}^\infty p_{d,k}k m_{1,k}=1.
\end{equation*}
When the second moment is finite, it follows that
\begin{equation*}
  \Var(Y_d)
  =
  \sum_{k=d+1}^\infty p_{d,k}k(k+1)m_{2,k}-1.
\end{equation*}
Consequently, whenever this variance is positive and finite, the shape parameter of the moment-matched unit-mean Gamma approximation is
\begin{equation}
  \label{eq:moment-matched-beta}
  \beta_d
  =
  \left[
  \sum_{k=d+1}^\infty p_{d,k}k(k+1)m_{2,k}-1
  \right]^{-1}.
\end{equation}

For every \(k\) with \(p_{d,k}>0\), the moments of the shape factor also admit an unnormalized configuration integral. Set $\lambda=1$ in Proposition~\ref{prop:configuration} and use the degree-zero homogeneity of $G/Q$ to obtain
\begin{equation}
  \label{eq:shape-moment-integral}
  m_{r,k}
  =
  \frac{\displaystyle
  \int_{\mathcal A_{d,k}}
  \left(\frac{G(\bfx)}{Q(\bfx)}\right)^r
  e^{-Q(\bfx)}\,\dd\bfx}
  {\displaystyle
  \int_{\mathcal A_{d,k}}e^{-Q(\bfx)}\,\dd\bfx}.
\end{equation}

The next proposition characterizes the conditional Gamma law whose shape is exactly the facet number. It does not exclude a nonconstant scale mixture from coinciding with a Gamma law of a different shape.

\begin{proposition}[Characterization of the shape-$k$ Gamma case and a concentration bound]
  \label{prop:gamma-characterization}
Fix $k$ with $p_{d,k}>0$.  The conditional law $Y_d\mid K=k$ is Gamma distributed with shape $k$ if and only if $A_k$ is almost surely constant.  More precisely, for $a>0$,
  \begin{equation}
    \label{eq:gamma-iff}
    Y_d\mid K=k
    \sim
    \operatorname{Gamma}(k,\operatorname{rate}=1/a)
    \quad\Longleftrightarrow\quad
    A_k=a\quad\text{almost surely}.
  \end{equation}
Moreover, with $\bar a_k=\E[A_k]$, for every $s\geq0$,
  \begin{equation}
    \label{eq:laplace-bound}
    0\leq
    \E[(1+sA_k)^{-k}]-(1+s\bar a_k)^{-k}
    \leq
    \frac{k(k+1)}2s^2\Var(A_k).
  \end{equation}
\end{proposition}

\begin{proof}
If $A_k=a$ almost surely, \eqref{eq:gamma-iff} follows immediately.  Conversely, equality with the stated Gamma law gives
  \begin{equation*}
    \E[Y_d\mid K=k]=ka,
    \qquad
    \E[Y_d^2\mid K=k]=k(k+1)a^2.
  \end{equation*}
Proposition~\ref{prop:transforms-moments} then implies $\E[A_k]=a$ and $\E[A_k^2]=a^2$, hence $\Var(A_k)=0$.

For the bound, put $\phi(a)=(1+sa)^{-k}$.  Since
  \begin{equation*}
    0\leq\phi''(a)
    =
    k(k+1)s^2(1+sa)^{-k-2}
    \leq k(k+1)s^2,
  \end{equation*}
Taylor's theorem about $\bar a_k$, followed by expectation, proves \eqref{eq:laplace-bound}.
\end{proof}

Proposition~\ref{prop:gamma-characterization} gives a precise explanation for the success and limitation of Gamma approximations with shape $k$. The conditional law has Gamma shape exactly $k$ if and only if the shape factor \(A_k\) is deterministic. A nonconstant scale mixture can in principle coincide with a Gamma law of a different shape, so the proposition makes no claim about that possibility. The concentration bound quantifies the effect of shape variability at the level of Laplace transforms. Obtaining analogous bounds for the densities themselves would require substantially stronger regularity information on the shape-factor laws.

If $Y_d\mid K=k$ is moment-matched to a Gamma distribution, its effective shape parameter is
\begin{equation}
  \label{eq:conditional-effective-shape}
  \beta_{k,\mathrm{eff}}
  =
  \frac{k}{1+(k+1)c_{A,k}^2},
  \qquad
  c_{A,k}^2
  =
  \frac{\Var(A_k)}{\E[A_k]^2}.
\end{equation}
Thus nonzero shape variability lowers the moment-matched conditional Gamma shape below $k$.

\subsection{Small-volume behaviour}
\label{Section:III-B}

The scale--shape factorization separates the lower tail of the typical-cell volume into combinatorial and geometric components.  The power of the leading term is determined by the smallest admissible facet number, \(d+1\), whereas its coefficient depends on a negative moment of the shape factor \(A_k\). We first translate the required negative moments into integrals over the Euclidean configuration sphere and give geometric conditions for their finiteness.  We then derive the small-volume limit.

Recall from Section~\ref{Section:II-C} that, with \(n=dk\),
\[
  \Theta_{d,k}
  =
  \mathcal A_{d,k}^{\rm reg}\cap\Sph^{n-1},
  \qquad
  \psi(\theta)=Q(\theta)^{-1/d}\theta,
\]
and
\begin{equation}
  \label{eq:shape-measure-angular-recall}
  \mu_{d,k}
  =
  \psi_\#\left(
    \frac1d Q(\theta)^{-k}
    \Haus^{n-1}\big|_{\Theta_{d,k}}
  \right).
\end{equation}
The next result shows that, at the critical negative order \(k\), the factors involving the flower volume cancel exactly.

\begin{proposition}[Angular representation of negative shape-factor moments]
  \label{prop:negative-moment-geometric-criterion}
Fix \(k\geq d+1\) with \(p_{d,k}>0\).  Then
  \begin{equation}
    \label{eq:negative-moment-angular}
    \E[A_k^{-k}]
    =
    \frac{1}{dkp_{d,k}}
    \int_{\Theta_{d,k}}
      G(\theta)^{-k}\,
      \dd\Haus^{dk-1}(\theta),
  \end{equation}
with equality in \([0,\infty]\).

Define
  \begin{equation}
    \label{eq:spherical-small-volume-function}
    M_{d,k}(\varepsilon)
    =
    \Haus^{dk-1}
    \bigl(
      \{\theta\in\Theta_{d,k}:G(\theta)\leq\varepsilon\}
    \bigr),
    \qquad \varepsilon>0.
  \end{equation}
If there exist \(C<\infty\), \(\delta>0\), and \(\varepsilon_0>0\) such that
  \begin{equation}
    \label{eq:spherical-small-volume-condition}
    M_{d,k}(\varepsilon)
    \leq C\varepsilon^{k+\delta},
    \qquad 0<\varepsilon\leq\varepsilon_0,
  \end{equation}
then
  \begin{equation}
    \label{eq:negative-moment-finite-geometric}
    \E[A_k^{-k}]<\infty.
  \end{equation}
\end{proposition}

\begin{proof}
By homogeneity and the definition of \(\psi\),
  \[
    \alpha(\psi(\theta))
    =
    G(\psi(\theta))
    =
    \frac{G(\theta)}{Q(\theta)}.
  \]
Using \eqref{eq:shape-measure-angular-recall}, we therefore obtain
  \begin{align*}
    \int_{\Sigma_{d,k}}\alpha(u)^{-k}\,\mu_{d,k}(\dd u)
    & =
    \frac1d
    \int_{\Theta_{d,k}}
      \left(\frac{G(\theta)}{Q(\theta)}\right)^{-k}
      Q(\theta)^{-k}\,
      \dd\Haus^{dk-1}(\theta) \\
    & =
    \frac1d
    \int_{\Theta_{d,k}}
      G(\theta)^{-k}\,
      \dd\Haus^{dk-1}(\theta).
  \end{align*}
Since \(A_k=\alpha(U_k)\), \(U_k\sim\nu_{d,k}\), and \(\nu_{d,k}=\mu_{d,k}/(kp_{d,k})\), this proves \eqref{eq:negative-moment-angular}.

It remains to control the integral near \(G=0\).  The contribution from \(\{G>\varepsilon_0\}\) is finite because \(\Theta_{d,k}\subset\Sph^{dk-1}\) has finite spherical measure.  On the remaining set, the layer-cake formula gives
  \begin{align}
    \int_{\{G\leq\varepsilon_0\}}
      G(\theta)^{-k}\,
      \dd\Haus^{dk-1}(\theta)
     = \varepsilon_0^{-k}M_{d,k}(\varepsilon_0)
        + k\int_0^{\varepsilon_0}
      \varepsilon^{-k-1}M_{d,k}(\varepsilon)\,
      \dd\varepsilon.
    \label{eq:negative-moment-layer-cake}
  \end{align}
Under \eqref{eq:spherical-small-volume-condition}, the last integral is at most
  \[
    kC\int_0^{\varepsilon_0}\varepsilon^{\delta-1}\,\dd\varepsilon
    =
    \frac{kC}{\delta}\varepsilon_0^\delta
    <\infty.
  \]
Equation~\eqref{eq:negative-moment-angular} now yields \eqref{eq:negative-moment-finite-geometric}.
\end{proof}

Proposition~\ref{prop:negative-moment-geometric-criterion} reduces the critical negative moment to a small-volume estimate on a finite-dimensional configuration sphere.  In particular, on the minimal-facet stratum \(k=d+1\), every candidate cell is a simplex, and
\begin{equation}
  \label{eq:minimal-facet-small-simplex-criterion}
  M_{d,d+1}(\varepsilon)
  =O(\varepsilon^{d+1+\delta})
  \quad\text{for some }\delta>0
\end{equation}
implies
\begin{equation}
  \label{eq:minimal-facet-negative-moment}
  \E[A_{d+1}^{-(d+1)}]<\infty.
\end{equation}

The same identity also makes the summability condition used below entirely geometric. For every $y_0>0$ and $m=d+1$,
\begin{align}
  &\sum_{k=m+1}^{\infty}
    \frac{p_{d,k}y_0^{k-m}}{\Gamma(k)}\E[A_k^{-k}]<\infty
  \notag\\
  &\qquad\Longleftrightarrow
  \sum_{k=m+1}^{\infty}
    \frac{y_0^{k-m}}{dk\Gamma(k)}
    \int_{\Theta_{d,k}}G(\theta)^{-k}\,
    \dd\Haus^{dk-1}(\theta)<\infty.
  \label{eq:geometric-negative-moment-sum}
\end{align}

The following criterion expresses \eqref{eq:minimal-facet-small-simplex-criterion}, and its analogue for general \(k\), in terms of the geometry of the degenerate-configuration locus.

\begin{corollary}[Tubular criterion]
  \label{cor:negative-moment-tubular-criterion}
Define the zero-degeneracy locus by
  \[
    \mathcal D_{d,k}
    =
    \bigcap_{j=1}^{\infty}
    \overline{\{\theta\in\Theta_{d,k}:G(\theta)<1/j\}},
  \]
where closures are taken in \(\Sph^{dk-1}\). This set is closed, and every sequence \((\theta_n)\subset\Theta_{d,k}\) satisfying \(G(\theta_n)\to0\) has all its accumulation points in \(\mathcal D_{d,k}\). Suppose that, for some \(c,C,\beta,q,r_0>0\),
    \begin{equation}
    G(\theta)
    \geq
    c\,\operatorname{dist}(\theta,\mathcal D_{d,k})^\beta,
    \label{eq:volume-distance-bound} 
  \end{equation}    
    \begin{equation}
    \Haus^{dk-1}
    \bigl(
      \{\theta\in\Theta_{d,k}:
      \operatorname{dist}(\theta,\mathcal D_{d,k})\leq r\}
    \bigr)
    \leq Cr^q,
    \qquad 0<r\leq r_0.
    \label{eq:tubular-volume-bound}
  \end{equation}
If \(q>\beta k\), then
  \[
    \E[A_k^{-k}]<\infty.
  \]
In particular, \(q>\beta(d+1)\) is sufficient on the minimal-facet stratum. The existence of exponents satisfying these inequalities is not asserted here.
\end{corollary}

\begin{proof}
If \(G(\theta)\leq\varepsilon\), then \eqref{eq:volume-distance-bound} implies
  \[
    \operatorname{dist}(\theta,\mathcal D_{d,k})
    \leq(\varepsilon/c)^{1/\beta}.
  \]
Hence \eqref{eq:tubular-volume-bound} gives, for all sufficiently small \(\varepsilon>0\),
  \[
    M_{d,k}(\varepsilon)
    \leq
    Cc^{-q/\beta}\varepsilon^{q/\beta}.
  \]
Since \(q/\beta>k\), Proposition~\ref{prop:negative-moment-geometric-criterion} applies.
\end{proof}

We now turn to the density of the normalized typical-cell volume \(Y_d=\lambda V_{d,\lambda}\).  The smallest possible facet number is \(d+1\), so its contribution has order \(y^d\).  The remaining issue is to justify passage to the limit through the infinite mixture over larger facet numbers. Notice that $p_{d,d+1}>0$: the outward normals of a regular simplex centered at the origin give a configuration in the interior of $\mathcal A_{d,d+1}$, and all sufficiently small perturbations remain admissible. The configuration integral in Proposition~\ref{prop:configuration} therefore assigns strictly positive mass to this open set.

\begin{proposition}[Small-volume limit under summability]
  \label{prop:small-volume}
Put \(m=d+1\). Suppose
  \begin{equation}
    \label{eq:negative-moment-first}
    \E[A_m^{-m}]<\infty
  \end{equation}
and that there exists \(y_0>0\) such that
  \begin{equation}
    \label{eq:negative-moment-sum}
    \sum_{k=m+1}^{\infty}
      \frac{p_{d,k}y_0^{k-m}}{\Gamma(k)}
      \E[A_k^{-k}]
    <\infty.
  \end{equation}
Indices with \(p_{d,k}=0\) are omitted from the sum.  Then
  \begin{equation}
    \label{eq:small-volume-limit}
    \lim_{y\downarrow0}\frac{f_{Y_d}(y)}{y^d}
    =
    \frac{p_{d,d+1}}{\Gamma(d+1)}
    \E[A_{d+1}^{-(d+1)}].
  \end{equation}
In particular, under these assumptions, if \(Y_d\) has a unit-mean Gamma distribution, then its shape parameter must be \(d+1\).
\end{proposition}

\begin{proof}
By intensity scaling and Corollary~\ref{cor:cdf-pdf},
  \begin{equation}
    \label{eq:normalized-mixture-density-small-volume}
    f_{Y_d}(y)
    =
    \sum_{k=m}^{\infty}
      \frac{p_{d,k}y^{k-1}}{\Gamma(k)}
      \E\left[
        A_k^{-k}\exp\left(-\frac{y}{A_k}\right)
      \right],
    \qquad y>0.
  \end{equation}
Divide \eqref{eq:normalized-mixture-density-small-volume} by \(y^{m-1}=y^d\).  For \(k=m\), the integrand
  \[
    A_m^{-m}\exp(-y/A_m)
  \]
converges almost surely to \(A_m^{-m}\) and is bounded by \(A_m^{-m}\).  Assumption \eqref{eq:negative-moment-first} and dominated convergence therefore give
  \[
    \lim_{y\downarrow0}
    \frac{p_{d,m}}{\Gamma(m)}
    \E\left[
      A_m^{-m}\exp\left(-\frac{y}{A_m}\right)
    \right]
    =
    \frac{p_{d,m}}{\Gamma(m)}\E[A_m^{-m}].
  \]

For \(k>m\) and \(0<y\leq y_0\), the corresponding rescaled term is bounded by
  \[
    \frac{p_{d,k}y_0^{k-m}}{\Gamma(k)}\E[A_k^{-k}].
  \]
These bounds are summable by \eqref{eq:negative-moment-sum}, whereas every fixed \(k>m\) term converges to zero as \(y \downarrow 0\).  Dominated convergence with respect to the counting measure on \(k\) now proves \eqref{eq:small-volume-limit}.
\end{proof}

The coefficient in \eqref{eq:small-volume-limit} has a clear structural interpretation.  The factor \(p_{d,d+1}\) is combinatorial: it is the probability that the Palm cell has the smallest possible number of facets.  The factor \(\E[A_{d+1}^{-(d+1)}]\) is geometric: by \eqref{eq:negative-moment-angular}, it measures the concentration of minimal-facet simplex configurations near zero candidate-cell volume.  Thus the scale--shape factorization separates the leading lower-tail coefficient into facet-number and shape contributions.

\begin{remark}[Status of the lower-tail assumptions]
  \label{rem:negative-moment-status}
Proposition~\ref{prop:negative-moment-geometric-criterion} converts the critical condition \eqref{eq:negative-moment-first} into an inverse-volume integral on a finite-dimensional configuration sphere. On the minimal-facet stratum, it is implied by the intensity-free simplex estimate \eqref{eq:minimal-facet-small-simplex-criterion}. Independently, \eqref{eq:geometric-negative-moment-sum} rewrites the higher-facet summability condition in geometric form.

Neither condition is established here. Corollary~\ref{cor:negative-moment-tubular-criterion} gives one sufficient criterion for a fixed stratum, but semialgebraicity alone does not imply the exponent inequality required there. The relevant degeneracies can be anisotropic and can intersect, so a proof appears to require a stratified local description in which the cell volume and spherical measure satisfy compatible product-type bounds. Moreover, proving the minimal-facet moment alone would not control the infinite sum over \(k>d+1\). Proposition~\ref{prop:small-volume} is therefore conditional on two logically separate geometric estimates.
\end{remark}

\section{Specializations and numerical illustration}
\label{sec:low-dimensional}

This section illustrates the general scale--shape factorization in several complementary ways. We first recover the one-dimensional Poisson--Voronoi cell, for which the shape factor and facet number are deterministic and the Gamma distribution follows immediately. We then derive an explicit coordinate representation for the planar case, obtaining finite-dimensional integral formulae for the area distribution. Finally, numerical experiments in dimensions one through four provide diagnostics for the principal consequences of the factorization and demonstrate Monte Carlo evaluation of the mixture representation.

\subsection{The one-dimensional case} 

The one-dimensional case is classical. Under the Palm distribution, the typical Poisson--Voronoi cell length is one half of the sum of the nearest-neighbour distances to the closest Poisson points on the left and on the right; see, e.g., \cite{Moller1994,Okabe2000}. Hence, it is Gamma-distributed with shape parameter \(2\) and rate \(2\lambda\). We now recover this classical fact as a degenerate instance of the scale--shape factorization.

A bounded Palm cell in one dimension has one effective neighbour on each side, so \(K=2\). Write the two distances as \(a,b>0\). For the ordered configuration \((-a,b)\),
\begin{equation*}
  P(-a,b)=\left[-\frac a2,\frac b2\right],
  \qquad
  G(-a,b)=\frac{a+b}{2}.
\end{equation*}
The flower is
\begin{equation*}
  \Delta(P(-a,b))=[-a,b],
  \qquad
  Q(-a,b)=a+b,
\end{equation*}
so the shape factor is deterministic:
\[
  \alpha=\frac{G}{Q}=\frac12 .
\]

The normalized surface $Q=1$ has two ordered branches,
\begin{align*}
  \Sigma_{1,2}^{(1)}&=\{(-s,1-s):0<s<1\},\\
  \Sigma_{1,2}^{(2)}&=\{(1-s,-s):0<s<1\}.
\end{align*}
On the first branch write $a=ts$ and $b=t(1-s)$, where $t>0$ and $0<s<1$. The Jacobian is $t$, so the radial decomposition assigns the measure $\dd s$ to the normalized branch. The same calculation applies to the second ordering. Hence each branch has shape-measure mass one, and therefore
\begin{equation*}
  \mu_{1,2}(\Sigma_{1,2})=2,
  \qquad
  p_{1,2}=1,
  \qquad
  \eta_{1,2}=\delta_{1/2}.
\end{equation*}
Corollary~\ref{cor:cdf-pdf} reduces to
\begin{align*}
  F_{V_{1,\lambda}}(v)
  &=1-(1+2\lambda v)e^{-2\lambda v},
  \\
  f_{V_{1,\lambda}}(v)
  &=4\lambda^2v e^{-2\lambda v}.
\end{align*}
Hence
\begin{equation*}
  V_{1,\lambda}\sim\operatorname{Gamma}(2,\operatorname{rate}=2\lambda),
\end{equation*}
which is also immediate from half the sum of the independent left and right nearest-neighbour distances.

The one-dimensional case illustrates the simplest possible shape-factor law. In two dimensions, the shape factor becomes genuinely random, although the general factorization still admits an explicit coordinate representation.

\subsection{Planar coordinate representation}
\label{sec:planar}

Set $d=2$ and write the $k$ effective neighbouring nuclei in cyclic order as
\begin{equation*}
  x_i=p_i e_i,
  \qquad
  p_i>0,
  \qquad
  e_i=(\cos\theta_i,\sin\theta_i)^T.
\end{equation*}
All indices in this subsection are interpreted modulo $k$.  Let
\begin{equation*}
  \delta_i=\theta_{i+1}-\theta_i
\end{equation*}
with cyclic interpretation and $\sum_i\delta_i=2\pi$.  Boundedness with nondegenerate adjacent supporting lines requires
\begin{equation*}
  0<\delta_i<\pi,
  \qquad i=1,\ldots,k.
\end{equation*}
The supporting line generated by $x_i$ is
\begin{equation*}
  L_i=\{z:\inner{z}{e_i}=p_i/2\}.
\end{equation*}
Writing $e_i^\perp=(-\sin\theta_i,\cos\theta_i)^T$, the vertex $v_i=L_i\cap L_{i+1}$ is obtained by solving
\begin{equation*}
  \inner{v_i}{e_i}=\frac{p_i}{2},
  \qquad
  \inner{v_i}{e_{i+1}}=\frac{p_{i+1}}{2}.
\end{equation*}
Since $e_{i+1}=\cos\delta_i e_i+\sin\delta_i e_i^\perp$, writing $v_i=(p_i/2)e_i+s_ie_i^\perp$ gives
\begin{equation*}
  \frac{p_i}{2}\cos\delta_i+s_i\sin\delta_i=\frac{p_{i+1}}2.
\end{equation*}
Thus
\begin{equation*}
  v_i
  =\frac{p_i}{2}e_i
  +\frac{p_{i+1}-p_i\cos\delta_i}{2\sin\delta_i}e_i^\perp.
\end{equation*}

The $i$th side lies on $L_i$ and has endpoints $v_{i-1}$ and $v_i$.  Their signed coordinates in the direction $e_i^\perp$ are
\begin{equation*}
  \frac{p_{i+1}-p_i\cos\delta_i}{2\sin\delta_i}
  \quad\text{and}\quad
  -\frac{p_{i-1}-p_i\cos\delta_{i-1}}{2\sin\delta_{i-1}},
\end{equation*}
respectively.  Therefore, the side length is
\begin{equation*}
  \ell_i
  =\frac{p_{i+1}-p_i\cos\delta_i}{2\sin\delta_i}
   +\frac{p_{i-1}-p_i\cos\delta_{i-1}}{2\sin\delta_{i-1}},
\end{equation*}
which is equivalent to
\begin{equation*}
  \ell_i
  =\frac{
  p_{i-1}\sin\delta_i
  +p_{i+1}\sin\delta_{i-1}
  -p_i\sin(\delta_{i-1}+\delta_i)}
  {2\sin\delta_{i-1}\sin\delta_i}.
\end{equation*}
The constraint indexed by $i$ is effective precisely when the relative interior of $F_i=P\cap L_i$ is nonempty, which in this cyclic nondegenerate representation is equivalent to $\ell_i>0$.  Hence all $k$ constraints are effective precisely when $\ell_i>0$ for every $i$; the boundary $\ell_i=0$ is a null degeneracy set.

Triangulation from the origin gives the polygon area.  The triangle with base $\ell_i$ on $L_i$ has height $p_i/2$, and therefore area $p_i\ell_i/4$.  Summing over sides yields
\begin{equation}
  \label{eq:planar-area}
  \mathcal G_k(\bfp,\bfd)
  =\frac14\sum_{i=1}^k p_i\ell_i.
\end{equation}

It remains to compute the flower area.  On the angular sector from $e_i$ to $e_{i+1}$, the support point of $P$ is $v_i$.  For
\begin{equation*}
  u(\varphi)=\cos\varphi\,e_i+\sin\varphi\,e_i^\perp,
  \qquad 0\leq\varphi\leq\delta_i,
\end{equation*}
we have
\begin{equation*}
  h_P(u(\varphi))
  =\inner{u(\varphi)}{v_i}
  =\frac{p_i\sin(\delta_i-\varphi)+p_{i+1}\sin\varphi}{2\sin\delta_i}.
\end{equation*}
Since $d=2$, formula \eqref{eq:flower-support} becomes
\begin{equation*}
  Q=2\int_0^{2\pi}h_P(u(\theta))^2\,\dd\theta.
\end{equation*}
A direct integration over $0\leq\varphi\leq\delta$ gives the sector contribution
\begin{equation}
  \Psi(\delta,p,q)
  =\frac{1}{2\sin^2\delta}
  \Bigg[
  \frac{\delta}{2}(p^2+q^2-2pq\cos\delta)
  +pq\sin\delta
  -\frac{p^2+q^2}{4}\sin(2\delta)
  \Bigg].
  \label{eq:planar-local-flower}
\end{equation}
Therefore
\begin{equation}
  \label{eq:planar-flower}
  \mathcal Q_k(\bfp,\bfd)
  =\sum_{i=1}^k\Psi(\delta_i,p_i,p_{i+1}).
\end{equation}

Let
\begin{equation*}
  S_k=\left\{\bfd\in[0,2\pi]^k:\sum_{i=1}^k\delta_i=2\pi\right\}
\end{equation*}
and let $\sigma_k$ be the normalized Lebesgue measure on this simplex.  Its $(k-1)$-dimensional coordinate volume is $(2\pi)^{k-1}/(k-1)!$.  Define
\begin{align*}
  \mathcal W_{k,\lambda}(\bfp,\bfd)
  &=\ind\{0<\delta_i<\pi\ \text{for every }i\}
  \left(\prod_{i=1}^k p_i\ind\{\ell_i>0\}\right)
  e^{-\lambda\mathcal Q_k(\bfp,\bfd)}.
\end{align*}

The angular factor is as follows.  In polar coordinates, $\dd x_i=p_i\,\dd p_i\,\dd\theta_i$.  The Mecke factor $1/k!$ first removes the $k!$ permutations of the same unordered set of angular directions.  After sorting the angles, the remaining integral is rotation invariant.  Passing from sorted angles to cyclic gaps quotients out the choice of the initial vertex; there are $k$ possible cyclic anchors.  Consequently, for any rotation-invariant cyclic integrand,
\begin{equation*}
  \frac1{k!}\int_{[0,2\pi)^k}F(\theta_1,\ldots,\theta_k)\,\dd\theta_1\cdots\dd\theta_k
  =\frac{2\pi}{k}\int_{S_k}F(\bfd)\,\dd\bfd.
\end{equation*}
Using the normalized measure $\sigma_k$ converts this factor into
\begin{equation*}
  \frac{2\pi}{k}\frac{(2\pi)^{k-1}}{(k-1)!}
  =\frac{(2\pi)^k}{k!}.
\end{equation*}
Here and below $\dd\bfp=\dd p_1\cdots\dd p_k$.  Consequently, for $a>0$,
\begin{align}
  F_{V_{2,\lambda}}(a)
  &=\sum_{k=3}^\infty\frac{(2\pi)^k\lambda^k}{k!}
  \int_{S_k}\int_{(0,\infty)^k}
  \ind\{\mathcal G_k(\bfp,\bfd)\leq a\}
  \mathcal W_{k,\lambda}(\bfp,\bfd)
  \,\dd\bfp\,\sigma_k(\dd\bfd).
  \label{eq:planar-cdf}
\end{align}
Its weak derivative is
\begin{align}
  f_{V_{2,\lambda}}(a)
  &=\sum_{k=3}^\infty\frac{(2\pi)^k\lambda^k}{k!}
  \int_{S_k}\int_{(0,\infty)^k}
  \delta_{\mathrm D}(a-\mathcal G_k(\bfp,\bfd))
  \mathcal W_{k,\lambda}(\bfp,\bfd)
  \,\dd\bfp\,\sigma_k(\dd\bfd),
  \label{eq:planar-pdf}
\end{align}
where $\delta_{\mathrm D}$ is the Dirac distribution and the identity is understood in the sense of distributions \cite[Chapter~9]{Folland1999}.

As a geometric check, a regular $k$-gon with inradius $r$ has
\begin{align*}
  G_k^{\mathrm{reg}}
  &=kr^2\tan\frac{\pi}{k},\\
  Q_k^{\mathrm{reg}}
  &=kr^2\sec^2\frac{\pi}{k}
  \left(\frac{2\pi}{k}+\sin\frac{2\pi}{k}\right),
\end{align*}
and therefore
\begin{equation}
  \label{eq:regular-ratio}
  A_k^{\mathrm{reg}}
  =\frac{\tfrac12\sin(2\pi/k)}{2\pi/k+\sin(2\pi/k)}
  \longrightarrow\frac14.
\end{equation}
This limit agrees with Lemma~\ref{lem:alpha-bound} for $d=2$.

The preceding analytical constructions lead directly to numerical evaluation. We next provide diagnostics for the principal structural consequences of the factorization and illustrate Monte Carlo evaluation of the resulting mixture representation.

\subsection{Numerical illustration and Monte Carlo mixture evaluation}

The numerical experiments illustrate, rather than prove, consequences of the scale--shape factorization. We use two distinct data sets. A reproducibility run is used for the conditional factorization diagnostics in Table~\ref{tab:factorization-diagnostics}; the original mixture sample is retained for Figure~\ref{fig:mixture-vs-generalized-gamma}. The first study examines the conditional Gamma law of $\lambda Q$, the monotone association between $\lambda Q$ and $A=V/Q$, and the support bound $0<A\leq2^{-d}$. The second substitutes simulated pairs $(K,A)$ into Corollary~\ref{cor:cdf-pdf} to estimate the marginal density of $Y_d=\lambda V_{d,\lambda}$. Spearman correlation detects only monotone association and is not, by itself, a test of independence. The simulation code, machine-readable data, and full reproducibility outputs are available from the public repository stated in the Data and Code Availability section.

\subsubsection{Simulation protocol and factorization diagnostics}

We set $\lambda=1$ for intensity scaling. The reproducibility run used Python 3.7.6 (64-bit), NumPy 1.21.6, and SciPy 1.7.3. Its master seed was $20260711$; the Palm-cell generator in dimension $d$ used \texttt{default\_rng(20260711+1000d)}. The full platform specification and all run settings are archived in the public repository as \path{results/reproducibility_study/run_summary.json}.

We generate the Poisson process shell by shell. At a radius stage $(R_{\rm in},R_{\rm out})$, the number of new nuclei is
\[
  N\sim\operatorname{Poisson}
  \left\{
    \lambda\kappa_d(R_{\rm out}^d-R_{\rm in}^d)
  \right\}.
\]
Directions are obtained by normalizing independent standard Gaussian vectors, and the radius is generated from
\[
  r^d=R_{\rm in}^d+U(R_{\rm out}^d-R_{\rm in}^d),
  \qquad U\sim\operatorname{Uniform}(0,1).
\]
The simulation retains previously generated nuclei as the simulation radius expands. The initial radius is
\[
  R_0=3.5(\lambda\kappa_d)^{-1/d}.
\]
After an unsuccessful stage, it is updated by $R\leftarrow1.25R$, with at most six expansions.

The origin cell is constructed with \texttt{scipy.spatial.Voronoi}, and its volume is computed with \texttt{scipy.spatial.ConvexHull}; both routines use the Qhull implementation bundled with SciPy. A cell is certified when
\begin{equation}
  \label{eq:numerical-certification}
  r_{\max}<R/2,
\end{equation}
where $r_{\max}$ is its largest vertex radius. The implementation uses the conservative test $r_{\max}<(R/2)(1-10^{-9})$. Indeed, if $\norm{y}\geq R$ and $\norm{z}<R/2$, then
\[
  \inner{z}{y}<\frac R2\norm{y}\leq\frac{\norm{y}^2}{2},
\]
so a nucleus outside $B(0,R)$ cannot alter the certified cell. An unsuccessful stage is not a rejected realization: we enlarge the same Poisson process and test it again.

The reproducibility run produced \(20\,000\), \(5000\), and \(5000\) certified cells in dimensions \(2\), \(3\), and \(4\), respectively, with no rejected or uncertified realizations. The expansion histogram and intermediate failure counts are reported in Appendix~\ref{app:certification-diagnostics}, Table~\ref{tab:expansion-summary}; machine-readable records are provided in \path{results/reproducibility_study/run_summary.json}.

In every dimension, $K$ is the number of distinct nuclei sharing a Voronoi ridge with the Palm nucleus, as returned by \texttt{Voronoi.ridge\_points}. Thus $K$ counts supporting nucleus constraints, not simplices used by Qhull to represent a polygonal facet. We independently checked this value using the vertex--constraint residual
\[
  e_{vj}=\inner{v}{p_j}-\frac{\norm{p_j}^2}{2}.
\]
Constraint $j$ was declared active when
\[
  |e_{vj}|\leq5\times10^{-8}
  \left(1+\frac{\norm{p_j}^2}{2}\right)
\]
for at least one computed vertex. We counted distinct tight constraints separately, even when their supporting hyperplanes were nearly coplanar. The residual and Voronoi-ridge counts agreed for all $30\,000$ certified cells.

The flower volume is evaluated from
\[
  Q=2^d\kappa_d\,\E[h_P(U)^d],
\]
where $U$ is uniform on $\Sph^{d-1}$. Sobol nets were generated with \texttt{qmc.Sobol(scramble=True)} and \texttt{random\_base2(m)}. Their coordinates were clipped to $[10^{-12},1-10^{-12}]$, transformed by the standard-normal inverse distribution function, and normalized to the sphere. The main diagnostic calculation used two independent Owen scrambles with $m=11$ (2048 directions) and seeds $20260711+100000d$ and $20260712+100000d$. If $Q_i^{(1)}$ and $Q_i^{(2)}$ are the two estimates, we use their average and record
\[
  \Delta_i
  =
  \frac{|Q_i^{(1)}-Q_i^{(2)}|}
       {(Q_i^{(1)}+Q_i^{(2)})/2}.
\]
The median and 95th percentile of $\Delta_i$ were $2.93\times10^{-4}$ and $8.23\times10^{-4}$ for $d=2$, $7.51\times10^{-4}$ and $2.28\times10^{-3}$ for $d=3$, and $1.83\times10^{-3}$ and $5.18\times10^{-3}$ for $d=4$. A separate four-scramble experiment on $128$ cells per dimension compared $2^8,\ldots,2^{13}$ directions; at the adopted $2048$-point resolution, the median relative errors were $7.28\times10^{-5}$, $2.77\times10^{-4}$, and $6.53\times10^{-4}$ in dimensions $2$, $3$, and $4$. The complete convergence table is archived in the public repository as \path{results/reproducibility_study/sobol_convergence.csv}.

We compute the diagnostics within facet-number strata. Let $n_d$ be the total number of certified cells, let $n_k$ be the number with $K=k$, and let $\widehat F_{\lambda Q\mid K=k}$ be the empirical conditional distribution function. Define
\[
  \widehat p_{d,k}=\frac{n_k}{n_d},
  \qquad
  D^{(Q)}_{d,k}
  =
  \sup_{q\geq0}
  \left|
    \widehat F_{\lambda Q\mid K=k}(q)
    -F_{\operatorname{Gamma}(k,\operatorname{rate}=1)}(q)
  \right|.
\]
Table~\ref{tab:factorization-diagnostics} reports strata with \(n_k\geq200\). Its final row in each dimensional block reports the combined mass of the omitted strata. The omitted masses are \(0.00895\), \(0.12020\), and \(0.44200\) for \(d=2,3,4\), respectively. All cells, including those from omitted strata, enter the unconditional mixture estimator. The complete stratum table is archived in the public repository as \path{results/reproducibility_study/facet_diagnostics.csv}.

\begin{table}[!t]
  \centering
  \caption{Direct diagnostics for the scale--shape factorization, grouped by dimension \(d\).
  Here \(\widehat p_{d,k}=n_k/n_d\) is the empirical facet-number probability,
  \(D^{(Q)}_{d,k}\) is the one-sample Kolmogorov--Smirnov distance between
  \(\lambda Q\mid\{K=k\}\) and
  \(\operatorname{Gamma}(k,\operatorname{rate}=1)\), and
  \(\rho_S(\lambda Q,A\mid K=k)\) is the conditional Spearman correlation between
  \(\lambda Q\) and the shape factor \(A\). Individual strata with
  \(n_k\geq200\) are displayed; the remaining strata are combined in the omitted-strata row.
  All strata, including omitted ones, are included in the unconditional mixture estimate. The \(d=1\)
  row is exact.}
  \label{tab:factorization-diagnostics}

  \scriptsize
  \setlength{\tabcolsep}{3.2pt}
  \renewcommand{\arraystretch}{0.94}

  \begin{tabular}{
    S[table-format=2.0]
    S[table-format=5.0]
    S[table-format=1.5]
    S[table-format=1.4]
    S[table-format=-1.4]
    S[table-format=1.6]
    S[table-format=1.4]
  }
    \toprule
    {$k$}
      & {$n_k$}
      & {$\widehat p_{d,k}$}
      & {$D^{(Q)}_{d,k}$}
      & {$\rho_S(\lambda Q,A\mid K=k)$}
      & {$\max A$}
      & {$\max A/2^{-d}$} \\
    \midrule

    \multicolumn{7}{c}{$d=1$} \\
    \midrule
    2
      & \multicolumn{1}{c}{exact}
      & 1.00000
      & 0.0000
      & \multicolumn{1}{c}{n/a}
      & 0.500000
      & 1.0000 \\

    \midrule
    \multicolumn{7}{c}{$d=2$} \\
    \midrule
     3 &  244 & 0.01220 & 0.0457 &  0.0700 & 0.145776 & 0.5831 \\
     4 & 2109 & 0.10545 & 0.0121 & -0.0331 & 0.192426 & 0.7697 \\
     5 & 5181 & 0.25905 & 0.0101 &  0.0054 & 0.206880 & 0.8275 \\
     6 & 5869 & 0.29345 & 0.0116 & -0.0073 & 0.217056 & 0.8682 \\
     7 & 3998 & 0.19990 & 0.0115 &  0.0217 & 0.225062 & 0.9002 \\
     8 & 1841 & 0.09205 & 0.0192 &  0.0119 & 0.228478 & 0.9139 \\
     9 &  579 & 0.02895 & 0.0424 &  0.0158 & 0.232606 & 0.9304 \\
    \multicolumn{2}{c}{omitted strata}
       & 0.00895
       & \multicolumn{4}{c}{---} \\

    \midrule
    \multicolumn{7}{c}{$d=3$} \\
    \midrule
    11 & 245 & 0.04900 & 0.0438 &  0.0872 & 0.075081 & 0.6006 \\
    12 & 370 & 0.07400 & 0.0561 &  0.0350 & 0.075720 & 0.6058 \\
    13 & 530 & 0.10600 & 0.0321 &  0.0535 & 0.080939 & 0.6475 \\
    14 & 538 & 0.10760 & 0.0385 &  0.0113 & 0.082276 & 0.6582 \\
    15 & 642 & 0.12840 & 0.0366 & -0.0210 & 0.085198 & 0.6816 \\
    16 & 577 & 0.11540 & 0.0464 & -0.0174 & 0.084423 & 0.6754 \\
    17 & 525 & 0.10500 & 0.0521 &  0.0295 & 0.089126 & 0.7130 \\
    18 & 416 & 0.08320 & 0.0315 & -0.0091 & 0.087405 & 0.6992 \\
    19 & 341 & 0.06820 & 0.0367 & -0.1061 & 0.085039 & 0.6803 \\
    20 & 215 & 0.04300 & 0.0728 &  0.0568 & 0.087265 & 0.6981 \\
    \multicolumn{2}{c}{omitted strata}
       & 0.12020
       & \multicolumn{4}{c}{---} \\

    \midrule
    \multicolumn{7}{c}{$d=4$} \\
    \midrule
    32 & 225 & 0.04500 & 0.0537 &  0.0230 & 0.029931 & 0.4789 \\
    33 & 234 & 0.04680 & 0.0684 &  0.0438 & 0.031196 & 0.4991 \\
    34 & 270 & 0.05400 & 0.0418 &  0.0160 & 0.031426 & 0.5028 \\
    35 & 273 & 0.05460 & 0.0617 & -0.0315 & 0.033753 & 0.5400 \\
    36 & 284 & 0.05680 & 0.0361 & -0.0027 & 0.033000 & 0.5280 \\
    37 & 266 & 0.05320 & 0.0381 & -0.0162 & 0.031633 & 0.5061 \\
    38 & 278 & 0.05560 & 0.0323 & -0.0107 & 0.033394 & 0.5343 \\
    39 & 254 & 0.05080 & 0.0606 & -0.0677 & 0.033440 & 0.5350 \\
    40 & 275 & 0.05500 & 0.0390 &  0.0035 & 0.034030 & 0.5445 \\
    41 & 216 & 0.04320 & 0.0543 & -0.0398 & 0.034115 & 0.5458 \\
    42 & 215 & 0.04300 & 0.0792 & -0.0084 & 0.034237 & 0.5478 \\
    \multicolumn{2}{c}{omitted strata}
       & 0.44200
       & \multicolumn{4}{c}{---} \\
    \bottomrule
  \end{tabular}
\end{table}

Across the displayed strata, the largest relative errors in the conditional mean identity \(\E[\lambda Q\mid K=k]=k\) were \(8.72\times10^{-3}\), \(2.60\times10^{-2}\), and \(1.69\times10^{-2}\) for \(d=2,3,4\), respectively. The largest KS distances were \(0.0457\), \(0.0728\), and \(0.0792\), and the largest absolute rank correlations were \(0.0700\), \(0.1061\), and \(0.0677\). These values are consistent with the predicted conditional Gamma law and show only weak monotone scale--shape association within the displayed strata. They do not by themselves establish independence. Because numerical error in \(Q\) enters both \(Q\) and \(A=V/Q\), the rank correlations should be interpreted as descriptive diagnostics. The largest values of \(A/2^{-d}\) among the displayed strata were \(0.9304\), \(0.7130\), and \(0.5478\) for \(d=2,3,4\), respectively, consistent with Lemma~\ref{lem:alpha-bound}.

\subsubsection{Numerical evaluation of the mixture density}

Since Corollary~\ref{cor:cdf-pdf} expresses the density as an expectation over the joint law of $(K, A)$, it admits the empirical approximation
\begin{equation}
  \label{eq:numerical-mixture-estimator}
  \widehat f^{\rm mix}_d(y)
  =
  \frac{1}{M_d}\sum_{i=1}^{M_d}
  \frac{y^{K_i-1}}{\Gamma(K_i)A_i^{K_i}}
  \exp\left(-\frac{y}{A_i}\right),
  \qquad y>0.
\end{equation}
Each summand is a $\operatorname{Gamma}(K_i,\operatorname{rate}=1/A_i)$ density, so $\widehat f^{\rm mix}_d$ is nonnegative and integrates to one. In one dimension, the expectation reduces exactly to $4y\exp(-2y)$.

Figure~\ref{fig:mixture-vs-generalized-gamma} is based on the certified reproducibility samples, with $M_2=20\,000$ and $M_3=M_4=5000$. The adaptive simulation begins at $R_0=3.5(\lambda\kappa_d)^{-1/d}$ and enlarges the simulation window when certification fails, as detailed in Appendix~\ref{app:certification-diagnostics}. We evaluate flower volumes using $2048$ scrambled Sobol directions. The density curves are evaluated at $2000$ equally spaced points over $10^{-4}\leq y\leq4.5$.

\begin{figure}[!t]
  \centering
  \includegraphics[width=0.75\linewidth]{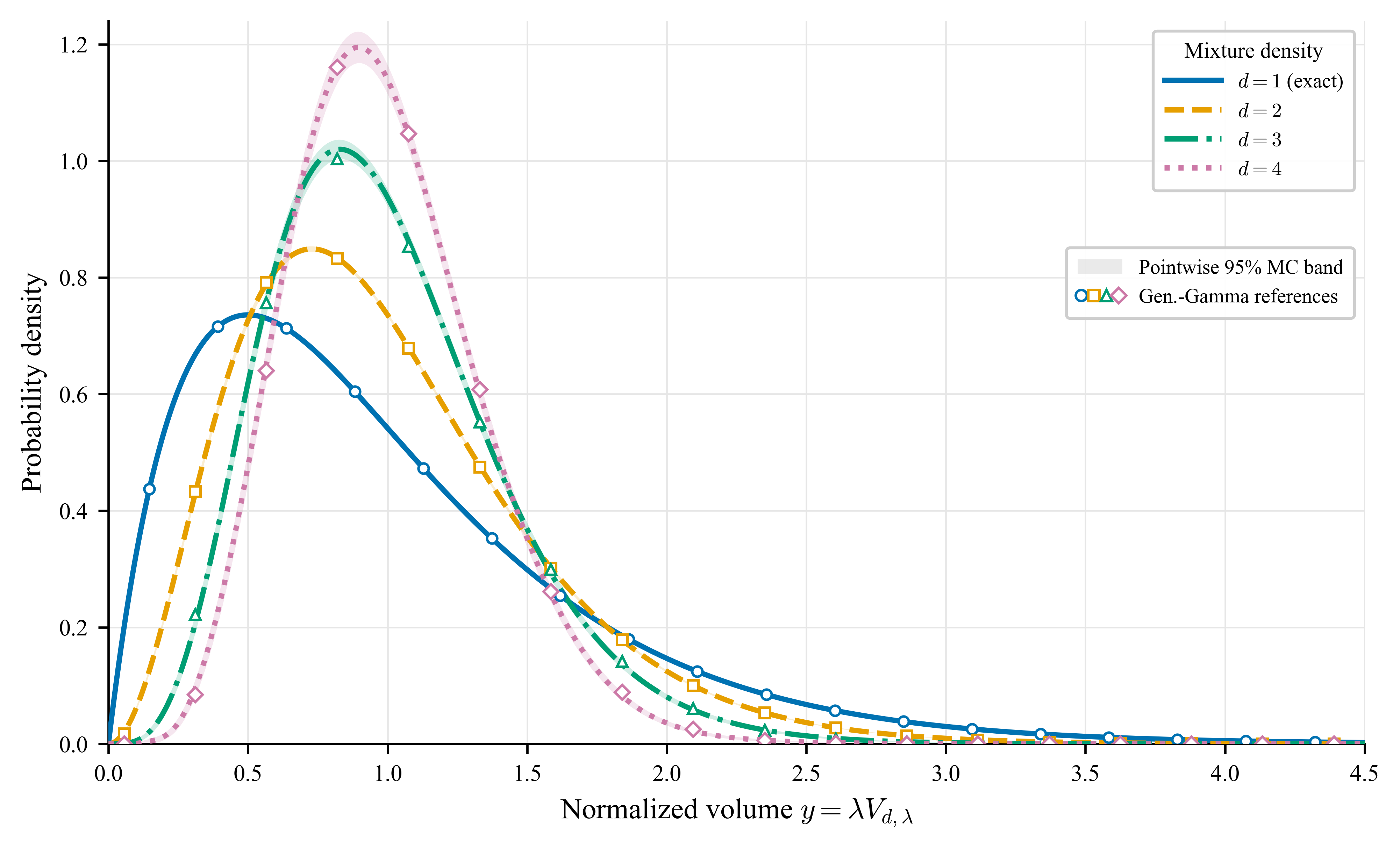}
  \caption{Monte Carlo plug-in estimates of the mixture density implied by the exact scale--shape representation (lines), with pointwise normal-approximation Monte Carlo bands (shading), and fixed generalized-Gamma reference densities (markers). The $d=1$ density is exact because the shape factor and facet number are deterministic. For $d=2,3,4$, the curves incorporate both shape variation and facet-number mixing. The bands, computed from the same samples as the corresponding curves, quantify finite Palm-cell sampling error; Sobol quadrature error is assessed separately. The reference densities are independently specified qualitative benchmarks.}
  \label{fig:mixture-vs-generalized-gamma}
\end{figure}

For comparison, we use the generalized-Gamma density
\begin{equation}
  \label{eq:numerical-generalized-gamma}
  h_d(y)
  =
  \frac{a_db_d^{c_d/a_d}}{\Gamma(c_d/a_d)}
  y^{c_d-1}\exp(-b_dy^{a_d}),
  \qquad y>0.
\end{equation}
The fixed parameters, taken from the numerical Poisson--Voronoi cell-size fits reported by Tanemura~\cite{TanemuraWeb}, are
\[
\begin{array}{c|ccc}
  d & a_d & b_d & c_d\\
  \hline
  1 & 1.00000 & 2.00000 & 2.00000\\
  2 & 1.07805 & 3.04011 & 3.31498\\
  3 & 1.16391 & 4.06342 & 4.80651\\
  4 & 1.29553 & 4.90493 & 6.49707
\end{array}
\]
These parameters are fixed independently of the Monte Carlo samples. For $d=1$, \eqref{eq:numerical-generalized-gamma} is the exact $\operatorname{Gamma}(2,\operatorname{rate}=2)$ density. For $d=2,3,4$, the reference densities serve only as qualitative parametric benchmarks.

At each grid point, we use the pointwise normal-approximation Monte Carlo band
\[
  \widehat f^{\rm mix}_d(y)
  \mathbin{\pm}
  1.96\,
  \frac{\operatorname{sd}\{g_1(y),\ldots,g_{M_d}(y)\}}{\sqrt{M_d}},
\]
truncated below at zero, where $g_i$ denotes the cell-level Gamma kernel in \eqref{eq:numerical-mixture-estimator}. These bands quantify finite Palm-cell sampling error; Sobol quadrature error is assessed separately in the convergence study above. The estimates and pointwise limits are archived in the public repository as \path{results/reproducibility_study/curve_d2.csv}, \path{results/reproducibility_study/curve_d3.csv}, and \path{results/reproducibility_study/curve_d4.csv}.

As a numerical normalization check, we combined trapezoidal integration over the displayed grid with the analytic left- and right-tail masses of every mixture component. The resulting totals for $d=2,3,4$ differed from one by at most $7\times10^{-10}$. Full-precision values are archived in the public repository as \path{results/reproducibility_study/normalization.csv}.

The exact agreement in one dimension checks the plotting and normalization conventions. In higher dimensions, the generalized-Gamma markers track the central part and displayed right tail of the Monte Carlo mixture estimates. This comparison is descriptive: we do not fit the reference parameters to the samples used here, and we make no inferential claim based on their visual agreement.

\section*{Data and Code Availability}

The original simulation and plotting code, the independent reproducibility study, cached mixture curves, cell-level data, facet diagnostics, Sobol convergence results, normalization checks, and run summaries are publicly available at \url{https://github.com/shitian-0715/poisson-voronoi-volume}.

\FloatBarrier

\section{Conclusion}
\label{sec:conclusion}

Building on the classical Gamma and complementary-theorem results for the Poisson--Voronoi fundamental region, we have developed an explicit configuration-space factorization of the Palm-typical cell volume. Conditional on the number of effective facets, the flower-volume scale is Gamma-distributed and independent of the normalized effective-neighbour configuration. Mapping that configuration to the cell-to-flower ratio expresses the cell volume as the product of the classical Gamma scale and a bounded shape factor. This representation yields exact mixture, transform, and moment formulae, characterizes the conditional Gamma case, and separates within-stratum shape variability from mixing across facet-number strata.

The same framework reduces the critical negative moments governing the lower tail to inverse-volume integrals over normalized configuration spaces. Under the stated integrability and higher-facet summability assumptions, the density of \(Y_d=\lambda V_{d,\lambda}\) satisfies
\[
  f_{Y_d}(y)
  \sim
  \frac{p_{d,d+1}}{\Gamma(d+1)}
  \E[A_{d+1}^{-(d+1)}]\,y^d,
  \qquad y\downarrow0.
\]
The one-dimensional distribution is recovered exactly, explicit planar coordinates are obtained, and simulations in dimensions two through four illustrate the mixture representation. Establishing the required geometric integrability and summability conditions remains the principal step toward an unconditional lower-tail result.

\appendix

\section{Simulation certification and expansion diagnostics}
\label{app:certification-diagnostics}

The reproducibility simulation begins with radius
\[
  R_0=3.5(\lambda\kappa_d)^{-1/d}.
\]
If a stage is unsuccessful, the radius is updated according to
\[
  R\leftarrow1.25R,
\]
with at most six expansions. At each expansion, we retain the existing Poisson configuration and supplement it with points generated in the newly added shell. Thus, an unsuccessful intermediate stage does not represent a rejected realization.

A cell is certified when
\[
  r_{\max}<R/2,
\]
where \(r_{\max}\) is the maximum distance from the nucleus to a cell vertex. The condition is implemented conservatively as
\[
  r_{\max}<(R/2)(1-10^{-9}).
\]
After certification, the flower volume is evaluated using two independently scrambled Sobol rules, each containing \(2048=2^{11}\) directions, and the two estimates are averaged.

The run produced \(20\,000\), \(5000\), and \(5000\) certified cells in dimensions \(2\), \(3\), and \(4\), respectively. No realization was finally rejected or remained uncertified, and none exhausted the six-expansion limit. Table~\ref{tab:expansion-summary} reports the expansion histogram and the total number of unsuccessful intermediate stages.

\begin{table}[!t]
  \centering
  \caption{Certification and adaptive-radius expansion summary. The column ``0 expansions''
  records cells certified at the initial radius; ``failed stages'' counts unsuccessful
  intermediate stages before eventual certification.}
  \label{tab:expansion-summary}
  \begin{tabular}{crrrrrrr}
    \toprule
    \(d\) & certified
      & \multicolumn{5}{c}{number of expansions}
      & failed stages \\
    \cmidrule(lr){3-7}
      & & 0 & 1 & 2 & 3 & 4 & \\
    \midrule
    2 & \(20\,000\) & \(10\,799\) & \(6387\) & \(2473\) & \(328\) & \(13\) & \(12\,369\) \\
    3 & \(5000\)    & \(2765\)    & \(2166\) & \(69\)   & \(0\)   & \(0\)  & \(2304\) \\
    4 & \(5000\)    & \(3895\)    & \(1105\) & \(0\)    & \(0\)   & \(0\)  & \(1105\) \\
    \bottomrule
  \end{tabular}
\end{table}

In dimension \(2\), the \(12\,369\) unsuccessful stages comprised \(11\,819\) failed certification tests, \(548\) temporarily unbounded origin regions, and \(2\) Qhull failures. All unsuccessful intermediate stages in dimensions \(3\) and \(4\) were failed certification tests. We archive machine-readable summary records and the complete run configuration in the public repository as \path{results/reproducibility_study/run_summary.json}.

\fund This research received no specific grant from any funding agency, commercial or not-for-profit sector.

\competing The authors declare no competing interests.

\bibliographystyle{apt}
\bibliography{MyRef.bib}

\end{document}